\documentclass{article}
\usepackage{gastex}
\usepackage{tikz, hyperref}
\usepackage{amsmath,amscd}
\usepackage{amssymb}
\usepackage{upref}
\usepackage{multirow}
\usepackage{multicol}
\usepackage{url}
\usepackage[all]{xy}

\usepackage{color}
\usepackage{lineno}
\usepackage{makeidx}
\makeindex
\usepackage{theorem}
\newtheorem{theorem}{Theorem}
\newtheorem{lemma}[theorem]{Lemma}
\newtheorem{proposition}[theorem]{Proposition}
\newtheorem{corollary}[theorem]{Corollary}
{\theorembodyfont{\rmfamily}%
  \newtheorem{example}[theorem]{Example}
   
\newtheorem{definition}[theorem]{Definition}}
\newenvironment{proof}{\noindent\textit{Proof.}}
{\QED\vskip\theorempostskipamount} 
\newenvironment{proofof}[1]{\noindent\textit{Proof
    \protect{#1}.}}
                       {\QED\vskip\theorempostskipamount}
\def\petitcarre{\vrule height4pt width 4pt depth0pt}
\def\QED{\relax\ifmmode\eqno{\hbox{\petitcarre}}\else{%
  \unskip\nobreak\hfil\penalty50\hskip2em\hbox{}\nobreak\hfil
  \petitcarre
  \parfillskip=0pt \finalhyphendemerits=0\par\smallskip}
  \fi}

\newcommand\A{\mathcal{A}}

\newcommand\cL{\mathcal{L}}

\newcommand{\GD}{\mathcal D}
\newcommand{\GH}{\mathcal H}
\newcommand{\GR}{\mathcal R}
\newcommand{\GL}{\mathcal L}

\newcommand{\N}{\mathbb{N}}
\newcommand{\Z}{\mathbb{Z}}

\def\un(#1){\underline{#1}\,}
\newcommand{\edge}[1]{\stackrel{#1}{\rightarrow}}
\newcommand{\longedge}[1]{\stackrel{#1}{\longrightarrow}}
\DeclareMathOperator{\Card}{Card}

\DeclareMathOperator{\rank}{rank}

\definecolor{ivoire}{rgb}{0.99,0.99,0.8}

\definecolor{light-gray}{gray}{0.7}
\usepackage{calc}
\newcounter{hours}\newcounter{minutes}

\numberwithin{equation}{section}
\numberwithin{figure}{section}
\numberwithin{table}{section}

\definecolor{lime}{HTML}{A6CE39}
\DeclareRobustCommand{\orcidicon}{%
	\begin{tikzpicture}
	\draw[lime, fill=lime] (0,0)
	circle [radius=0.16]
	node[white] {{\fontfamily{qag}\selectfont \tiny ID}};
	\draw[white, fill=white] (-0.0625,0.095)
	circle [radius=0.007];
	\end{tikzpicture}
	\hspace{-2mm}
}
\foreach \x in {A, ..., Z}{%
 \expandafter\xdef\csname orcid\x\endcsname{\noexpand%
 \href{https://orcid.org/\csname orcidauthor\x\endcsname}{\noexpand\orcidicon}}
}

\title{Recognizability of morphisms}

\author{Marie-Pierre B\'eal\orcidA{}%
and Dominique Perrin and Antonio Restivo}
\begin{document}

\maketitle
\begin{abstract}
 We investigate several questions related to the notion
  of recognizable morphism.
  The main result is a new proof of Moss\'e's theorem
  and actually of a generalization to a more general class of morphisms
  due to Berth\'e et al. We actually prove the result of Berth\'e et
  al. for the most general class of morphisms, including ones with
  erasable letters.
  Our result is derived from a result concerning elementary morphisms
  for which we also provide a new proof.
  We also prove some new results which allow us to formulate
  the property of recognizability in terms of finite automata.
  We use this characterization to
  show that, for  an injective morphism, the property of being
  recognizable on the full shift for aperiodic points is decidable.
\end{abstract}

\section{Introduction}
In this paper, we study  problems related to
the recognizability of morphisms (also called substitutions). This notion
concerns the unambiguity in the representation
of an infinite sequence $y$ as the image $\sigma(x)$ by a morphism $\sigma$
of another sequence $x$, up to some normalized shift,
called a centered $\sigma$-representation.  It belongs
to a general notion of unambiguity in symbolic
dynamics (see \cite{BealPerrinRestivo2021}).

By Moss\'e's Theorem (\cite{Mosse1992,Mosse1996}), every aperiodic primitive morphism $\sigma$
is recognizable on the shift $X(\sigma)$
(see the precise definitions in Section~\ref{sectionRecognizability}). This surprising
result was initially formulated (in an incorrect 
way) by \cite{Martin1973} (see \cite{Kyriakoglou2019} on the genesis
of the theorem and its possible variants). It was further generalized
by Bezuglyi et al. \cite{BezuglyiKwiatkowskiMedynets2009},
who proved that every aperiodic non-erasing morphism $\sigma$
is recognizable on $X(\sigma)$. Next, it was proved 
by Berth\'e et al. \cite{BertheSteinerThuswaldnerYassawi2019}
that every non-erasing morphism $\sigma$ is recognizable on $X(\sigma)$
for aperiodic points.

The main result of this paper is a generalization of Moss\'e's
Theorem.
\begin{theorem}
  Every morphism $\sigma$ is recognizable on $X(\sigma)$
  for aperiodic points.
\end{theorem}
This means that every aperiodic point  in $X(\sigma)$
has a unique centered $\sigma$-represen\-tation as a shift
of the image by $\sigma$ of some $x\in X(\sigma)$.

We actually prove the result of
\cite{BertheSteinerThuswaldnerYassawi2019} concerning
the recognizability for aperiodic points for the most general class
of morphisms, including ones containing erasable letters. Our proof does not use the fact that
the shift $X(\sigma)$ defined for an aperiodic non-erasing morphism has a finite number of
minimal subshifts (proved and used in
\cite{BezuglyiKwiatkowskiMedynets2009}), or the stronger fact that the set of
languages of points in a shift $X(\sigma)$ defined by a non-erasing morphism is
finite (proved and used in \cite{BertheSteinerThuswaldnerYassawi2019}).
Our
proof relies on the notion of elementary morphism,
due to Ehrenfeucht and Rozenberg \cite{EhrenfeuchtRozenberg1978}.
By a result, proved independently by Karhumaki et al.
\cite{KarhumakiManuchPlandowski2003}
and by Berth\'e et al. \cite{BertheSteinerThuswaldnerYassawi2019}, every
elementary morphism is recognizable for aperiodic points.
We use this result to prove Moss\'e's Theorem (and its generalizations).
This represents actually
also a substantial simplification of the proof, since the
available proofs of Moss\'e's Theorem are all 
more difficult, including the proof in \cite{Kurka2003}
which is essentially reproduced in \cite{DurandPerrin2021},
using improvements from \cite{Kyriakoglou2019}.

We also give a characterization of injective
morphisms recognizable on the full shift in terms of groups
in finite monoids (Theorem \ref{theoremSyntactic2}). This
translation of the property of recognizability in terms 
of finite monoids is a new approach which puts in evidence
the strong link between recognizability and automata theory.

Finally, we give a quadratic-time
algorithm to check whether an
injective morphism is recognizable on the full shift for aperiodic
points (Corollary \ref{corollary.recognizable}).

Note that the notion of recognizability can also be defined for substitutions
in dimension $2$ or higher. This was first considered in
\cite{Mozes1989} where the recognizability property is used to prove
that a $2\text{D}$ substitution shift is sofic. The property was also
used in \cite{AubrunSablik2013} to simulate an effective subshift by a
sofic and also in \cite{AubrunSablik2014}.
A generalization of Moss\'e's Theorem to the $2\text{D}$ case was proved in
\cite{Solomyak1998}. Let us also mention that the recognizability of morphisms has been extended to sequences
of morphisms in \cite{BertheSteinerThuswaldnerYassawi2019} (see also \cite{DonosoDurandMaassPetite2021}).

The paper is organized as follows. After an introductory section
on the basic notions of symbolic dynamics, we formulate
the precise definition of a morphism recognizable
on a shift space and prove some elementary properties
of recognizable morphisms. In Section \ref{sectionElementary}
we introduce elementary morphisms and prove that every
elementary morphism is recognizable for
aperiodic points (Theorem~\ref{theoremIndecomposable}).
In Section~\ref{sectionIterated}, we give the new proof
of Moss\'e's Theorem and its extensions. In Section~\ref{sectionSyntactic},
we formulate a condition characterizing
recognizable injective morphisms in terms of groups
in finite monoids (Theorem~\ref{theoremSyntactic2}).
In Section~\ref{section.efficient}, we show
that all notions handled in this paper are decidable
and, notably, by algorithms of low polynomial complexity.

\numberwithin{theorem}{section}
\section{Symbolic dynamics}
We briefly recall the definitions of symbolic dynamics. For
a more complete presentation, see~\cite{LindMarcus2021}
or the recent~\cite{DurandPerrin2021}.
\subsection{Words}
Let $A$ be a finite alphabet. We denote by $A^*$ the free
monoid on $A$ and by $A^+$ the set of non-empty words on $A$.
The empty word is denoted by $\varepsilon$.
We denote by $|u|$ the \emph{length} of the word $u$.

For a word $w\in A^*$ and $a\in A$, we denote by $|w|_a$
the number of occurrences of $a$ in $w$.

A word $s\in A^*$ is a \emph{factor}
of $w\in A^*$ if $w=rst$. The word $r$ is called a
\emph{prefix} of $w$. It is \emph{proper} if $r\ne w$
(that is, if $st$ is not empty).

For $U\subset A^*$, we denote by $U^*$ the submonoid of
$A^*$ generated by $U$. For $U=\{u\}$, we write $u^*$
instead of $\{u\}^*$.

A set $U\subset A^*$ is a \emph{code} if every
$w\in U^*$ has a unique decomposition in words of $U$.
A \emph{prefix code} is
a set $U\subset A^*$ such that no element of $U$ is a prefix
of another one.

Two words $u,v$ are \emph{conjugate} if $u=rs$ and $v=sr$
for some words $r,s$.

A word is \emph{primitive} if it is not a power of a shorter word.

An integer $p\ge 1$ is a \emph{period} of a word $w=a_0a_1\cdots a_{n-1}$
with $a_i\in A$
if $a_i=a_{i+p}$ for $0\le i\le n-p-1$.
\subsection{Shift spaces}
We consider the set $A^\Z$ of two-sided infinite sequences on $A$.
For $x=(x_n)_{n\in\Z}$, and $i\le j$, we denote
$x_{[i,j]}$ the word $x_ix_{i+1}\cdots x_j$
and $x_{[i,j)}$ the word $x_ix_{i+1}\cdots x_{j-1}$.

The set $A^\Z$ is a compact metric space for the distance defined for $x\ne y$
by $d(x,y)=2^{-r(x,y)}$ with
\begin{displaymath}
r(x,y)=\min\{|n|\mid n\in\Z, x_n\ne y_n\}.
\end{displaymath}
The \emph{shift transformation} $S:A^\Z\to A^\Z$ is defined
by $y=S(x)$ if $y_n=x_{n+1}$ for every $n\in\Z$.
We sometimes denote $Sx$ instead of $S(x)$.

A \emph{shift space} $X$ on a finite alphabet $A$ is a closed
and shift invariant subset of $A^\Z$.

A topological dynamical system is a pair $(X,S)$
of a compact metric space and a continuous map from $X$
to itself. For every shift space $X$, since $S$ is  continuous
and $X$ is compact, the pair $(X,S)$ is
a topological dynamical system.

The \emph{orbit} of a sequence $x\in A^\Z$ is the set
$\{S^n(x)\mid n\in \Z\}$.

A point $x\in A^\Z$ is \emph{periodic} if there is an $n\ge 1$
such that $S^n(x)=x$. Otherwise, it is \emph{aperiodic}.
A periodic point has the form $w^\infty=\cdots ww\cdot ww\cdots$
(the letter of index $0$ of $w^\infty$ is the first letter of $w$).
For $x\in A^\Z$, we denote $x^+=x_0x_1\cdots$, which is an element
of $A^\N$ and $x^-=\cdots x_{-2}x_{-1}$ which is the element $y$ of $A^{-\N}$
defined by $y_{-n}=x_{-n-1}$ for $n\in\N$.
For $x\in A^{-\N}$ and $y\in A^{\N}$, we denote $z=x\cdot y$
the two-sided infinite sequence $z$ such that $z^-=x$ and $z^+=y$.

A word $w$ is a factor of $x\in A^\Z$ if $w=x_ix_{i+1}\cdots x_{j-1}$
for some $i,j\in\Z$ with $i\le j$.

The \emph{language} of a shift space $X$, denoted $\cL(X)$,
is the set of factors of the elements of $X$.
For $x\in A^\Z$, we also denote by $\cL(x)$ the set
of factors of $x$. Thus $\cL(X)=\cup_{x\in X}\cL(x)$.

\subsection{Morphisms}

A morphism $\sigma\colon A^*\to B^*$ is a monoid morphism from $A^*$ to $B^*$.
It is \emph{non-erasing} if $\sigma(a)$ is nonempty for every $a\in
A$.

The \emph{incidence matrix} of a morphism $\sigma\colon A^*\to B^*$
is the $(A\times B)$-matrix $M(\sigma)$ whose row of index
$a$ is the vector $(|\sigma(a)|_b)_{b\in B}$.

For example, if $\sigma\colon a\mapsto ab,b\mapsto ac,c\mapsto a$,
we have
\begin{displaymath}
  M(\sigma)=\begin{bmatrix}1&1&0\\1&0&1\\1&0&0\end{bmatrix}
  \end{displaymath}

A morphism $\sigma$ defines a map from $A^\N$ to $B^* \cup B^\N$
defined by $\sigma(x_0x_1\cdots)=\sigma(x_0)\sigma(x_1)\cdots$,
from $A^{-\N}$ to $B^* \cup B^{-\N}$ defined by
$\sigma(\cdots x_{-1}x_{0})=\cdots\sigma(x_{-1})\sigma(x_0)$
and also a map from $A^\Z$ to $B^* \cup B^\N \cup B^{-\N} \cup B^\Z$
defined by $\sigma(x)=\sigma(x^-)\cdot\sigma(x^+)$.

Let $\sigma\colon A^*\to B^*$ be a morphism whose restriction to
$A$ is injective.
The set $U=\sigma(A)$ is a code if and only if $\sigma$ is injective
on $A^*$.
Note that $\sigma$ can be injective on $A^*$ without being
injective on $A^\Z$ (the converse being obviously true).
It is injective on $A^\Z$
if and only if it is injective on $A^\N$ and on $A^{-\N}$.
\begin{example}
  The Fibonacci morphism $a\mapsto ab,b\mapsto a$ is injective
  on $A^\N$, as one may easily verify.
  The morphism $\sigma\colon a\mapsto a,b\mapsto ab,c\mapsto bb$
  is not injective on $A^\N$ since $\sigma(ac^\omega)=\sigma(bc^\omega)$.
\end{example}

Let $\sigma\colon A^*\to A^*$ be a morphism from $A^*$ to itself.
The \emph{language} of $\sigma$, denoted $\cL(\sigma)$
is the set of factors of the words $\sigma^n(a)$ for
some $n\ge 0$ and $a\in A$.
The \emph{shift defined} by $\sigma$, denoted by $X(\sigma)$
is the set of sequences with all their factors in $\cL(\sigma)$.

One has $\cL(X(\sigma))\subset\cL(\sigma)$
but it is not true in general that $\cL(X(\sigma))=\cL(\sigma)$.
Indeed, the words of $\cL(X(\sigma))$ can always be extended
(we have $awb\in \cL(X(\sigma))$ for some $a,b\in A$)
while this needs not be true of the words of $\cL(\sigma)$.
Note that if $w\in \cL(X(\sigma))$, then $w$ is a factor
of some $\sigma^n(a)$ for all large enough $n$.

An \emph{erasable} letter of a morphism $\sigma\colon A^*\to A^*$ is a letter $a$ in
$A$ such that $\sigma^n(a) = \varepsilon$ for some integer $n$. A word
is \emph{erasable} if it is formed of erasable letters.

If $w$ is an erasable word, then $\sigma^{\Card(A)}(w)=\varepsilon$.
Indeed, set $A_i=\{a\in A\mid \sigma^i(a)=\varepsilon\}$. Then
$A_1\subset A_2\subset\ldots\subset A$ and thus there is $k\le\Card(A)$
such that $A_k=A_{k+1}$, which implies $A_{k+j}=A_{k}$ for all $j\ge 0$.
\begin{proposition}\label{propositionErasing}
  Let $\sigma\colon A^*\to A^*$ be a morphism. The set of erasable words in $\cL(\sigma)$ is finite.
\end{proposition}
We first prove the following well-known lemma
(see, for example, \cite[Lemma 5.5]{Queffelec2010}).
\begin{lemma}\label{lemmaQueffelec}
 For every  word $u\in \cL(\sigma)$, there is some $m\ge 0$
and words $v_i,w_i\in \cL(X)$ for $0\le i\le m$, such that
\begin{equation}
u=v_0\sigma(v_1)\cdots\sigma^{m-1}(v_{m-1})\sigma^m(v_m)\sigma^{m-1}(w_{m-1})\cdots\sigma(w_1)w_0,
\label{eqQueffelec}
\end{equation}
with $|v_i|,|w_i|\le |\sigma|$.
\end{lemma}
\begin{proof}
We use induction on $|u|$. The result is true if $|u|<|\sigma|$
choosing $m=0$ and $v_0=u$.
Otherwise, by definition of $\cL(\sigma)$, there exists a 
word $u'\in \cL(\sigma)$ such that
$u=v_0\sigma(u')w_0$. Choosing $u'\in\cL(\sigma)$ of maximal length, we have moreover 
$|v_0|,|w_0|\le|\sigma|$. By induction hypothesis, we have
a decomposition \eqref{eqQueffelec} for $u'$, that is
\begin{displaymath}
u'=v'_0\sigma(v'_1)\cdots\sigma^{m-1}(v'_{m-1})\sigma^m(v'_m)\sigma^{m-1}(w_{m-1}')\cdots\sigma(w_1')w_0'.
\end{displaymath}
In this way, we obtain
\begin{displaymath}
u=v_0\sigma(u')w_0=v_0\sigma(v'_0)\cdots\sigma^{m}(v'_{m-1})\sigma^{m+1}(v'_m)\sigma^{m}(w'_{m-1})\cdots\sigma(w_0')w_0.
\end{displaymath}
which is of the form \eqref{eqQueffelec}.
\end{proof}
\begin{proofof}{of Proposition~\ref{propositionErasing}}
  Let $u\in\cL(\sigma)$ be erasable. Let $m\ge 0$, $v_i,w_i$ with $|v_i|,|w_i|\le|\sigma|$
  be such that \eqref{eqQueffelec} holds. Since $u$ is erasable,
  all $v_i,w_i$ are erasable. Since $\sigma^{\Card(A)}(w)=\varepsilon$
  for every erasable word $w$, we can assume that $m\le\Card(A)-1$.
  This implies that the length of $u$ is bounded.
\end{proofof}
Note that the maximal length of erasable words in $\cL(\sigma)$
is, according to the proof above, bounded by
$2|\sigma|\sum_{i=0}^{k-1}|\sigma|^i=2|\sigma|(|\sigma|^k-1)/(|\sigma|-1)$
where $k=\Card(A)$.

\begin{corollary}\label{lemmaSigma}
  Let $\sigma\colon A^*\to A^*$ be a morphism. For every $x$ in $X(\sigma)$,
  the sequence $\sigma(x)$ is in $X(\sigma)$. The map
  $x\mapsto\sigma(x)$ is continuous on $X(\sigma)$.
\end{corollary}
\begin{proof}
 We have to prove that $\sigma(x)$ is two-sided infinite.
  By Proposition~\ref{propositionErasing}, $x$ has an infinite number of non-erasable letters on
  the left and on the right and thus $\sigma(x)$ is two-sided infinite.
  Since $\sigma$ is continuous at every point $x$ such that
  $\sigma(x)$ is infinite, it is continuous on $X(\sigma)$.
\end{proof}
Note that, since $X=X(\sigma)$ is compact, and
since $\sigma$ is continuous on $X$
by Corollary \ref{lemmaSigma}, the pair $(X,\sigma)$
is a topological dynamical system. Thus it is both a dynamical system
$(X,S)$ with respect to the shift transformation and
a dynamical system $(X,\sigma)$.

As another corollary of Proposition~\ref{propositionErasing}, we have the
following characterization of the morphisms such that
$\cL(X(\sigma))=\cL(\sigma)$.

\begin{proposition}\label{propositionCaractSubstitution}
  Let $\sigma\colon A^*\to A^*$ be a morphism.
  One has $\cL(\sigma)=\cL(X(\sigma))$
  if and only if every letter $a\in A$ is in $\cL(X(\sigma))$.
  \end{proposition}
\begin{proof}
  Assume that the condition is satisfied. Every $w\in \cL(\sigma)$ is a factor
  of some $\sigma^n(a)$ for $a\in A$ and $n\ge 0$. Let $x\in X(\sigma)$
  be such that $a\in\cL(x)$. By Corollary~\ref{lemmaSigma}, we have
  $\sigma^n(x)\in X(\sigma)$. Since $w\in\cL(\sigma^n(x))$, we obtain
  $w\in \cL(X(\sigma))$. The converse is obvious.
  \end{proof}

  \begin{example}
    The morphism $\sigma\colon a\mapsto ab,b\mapsto abab$ is periodic.
    The closure of $\sigma(A^\Z)$ is $(ab)^\infty\cup(ba)^\infty$.
    \end{example}
A word $w\in A^*$ is \emph{growing}
for $\sigma$ if the sequence $(|\sigma^n(w)|)_n$ is unbounded.
A word is growing if some of its letters is growing. The
morphism $\sigma$ itself is said to be \emph{growing} if all letters
are growing.

Actually, we have the following property of growing letters.
\begin{proposition}\label{lemmaGrowing}
 If
 $a\in A$ is  growing for $\sigma$, then $\sigma^{r\Card(A)}(a)$
 contains for every $r\ge 0$ at least $r+1$ non-erasable letters.
 In particular, $\lim_{n\to\infty}|\sigma^n(a)|=\infty$.

 \end{proposition}
\begin{proof}
  Set $k=\Card(A)$.
  Assume first that $\sigma^k(a)$ contains only one non-erasable letter.
  Then, this letter
  has to be growing.
  Next, by the pigeonhole principle, there are $i,p$ with
  $i+p\le k$ and $p\ge 1$ such that $\sigma^i(a)=ubv$ and $\sigma^p(b)=rbs$
  with $u,v,r,s$ erasable and $b$ a growing letter. But then,
  since $\sigma^{\Card(A)}(r)=\sigma^{\Card(A)}(s)=\varepsilon$, the word
  \begin{displaymath}
    w=\sigma^{p\Card(A))}(b)=\sigma^{p(\Card(A)-1)}(r)\cdots\sigma^p(r)rbs\sigma^p(s)\cdots\sigma^{p(\Card(A)-1)}(s)
  \end{displaymath}
  is a finite
  fixed point of $\sigma^p$, that is, such that
  $\sigma^p(w)=w$, a contradiction with the fact that
  $b$ is growing. This proves the statement for $r=1$.

  Next,
  assume that $\sigma^{rk}(a)$ contains $s\ge r+1$ non-erasable
  letters $a_1,\ldots, a_s$. One of them, say $a_i$, has to be growing.
  Then each of the $\sigma(a_1),\ldots,\sigma(a_s)$ contains
  a non-erasing letter and $\sigma^k(a_i)$ contains at least two
  by the property for $r=1$. Thus $\sigma^{(r+1)k}(a)$ contains at least
  $r+2$ non-erasing letters.

\end{proof}
\section{Recognizable morphisms}\label{sectionRecognizability}
Let $\sigma\colon A^*\to B^*$ be a morphism. A \emph{$\sigma$-representation}
of $y\in B^\Z$ is a pair $(x,k)$ of a sequence $x\in A^\Z$
and an integer $k$  such that
\begin{equation}
  y=S^k(\sigma(x))\label{eqsigmaRep}
\end{equation}
where $S$ denotes the shift transformation.
The $\sigma$-representation $(x,k)$ is \emph{centered}
if $0\le k<|\sigma(x_0)|$.

Note that, in particular, a centered $\sigma$-representation $(x,k)$ is
such that $\sigma(x_0)\ne\varepsilon$. 

Note that if $y$ has a $\sigma$-representation
$(x,k)$, it has a centered $\sigma$-representation $(x',k')$
with $x'$ a shift of $x$.
Indeed, assume $k\ge 0$ (the case $k<0$ is symmetric).
Let $i\ge 0$ be such that $|\sigma(x_0\cdots x_{i-1})|\le k<|\sigma(x_0\cdots x_i)|$. Set
$k'=k-|\sigma(x_0\cdots x_{i-1})|$ and $x'=S^ix$.
Then $S^{k'}\sigma(x')=S^{k'+|\sigma(x_0\cdots x_{i-1})|}\sigma(x)=S^k\sigma(x)=y$
and $0\le k'<|\sigma(x'_0)|$. Thus
$(x',k')$ is a centered $\sigma$-representation of $y$.

For a shift space $X$ on $A$, the set of points in $B^\Z$
having a $\sigma$-representation
$(x,k)$ with $x\in X$ is a shift space on $B$ which is the closure under the shift of $\sigma(X)$. Indeed, if $(x,k)$ is the $\sigma$-representation of $y$,
then $S(y)$ has the $\sigma$-representation $(x',k')$ with
\begin{displaymath}
  (x',k')=\begin{cases}(x,k+1)&\mbox{ if $k+1<|\sigma(x_0)|$}\\
    (S(x),0)&\mbox{ otherwise.}\end{cases}
    \end{displaymath}
\begin{definition}
Let $X$ be a shift space on $A$.
A morphism $\sigma\colon A^*\to B^*$ is \emph{recognizable} on $X$
(resp. recognizable on $X$
for aperiodic points) if for every point $y\in B^\Z$ (resp.
every aperiodic
point $y\in B^\Z$) there is at most one centered $\sigma$-representation $(x,k)$ of $y$
with $x\in X$.
\end{definition}
For $X=A^\Z$, we simply say recognizable instead of recognizable on $X$.
Note that, as an equivalent definition of recognizability on $X$
(resp. recognizability for aperiodic points),
for every $y,y'\in A^\Z$ and $0\le k<|\ |\sigma(y'_0)|-|\sigma(y_0)|\ |$ such that
$\sigma(y)=S^k(\sigma(y'))$ is in $X$
(resp. is an aperiodic point in $X$), one has $k=k'$ and $y=y'$.

A morphism $\sigma\colon A^*\to B^*$ is \emph{circular} if
 it is injective and if for every $u,v\in B^*$
\begin{displaymath}
  uv,vu\in\sigma(A^*)\Rightarrow u,v\in\sigma(A^*).
\end{displaymath}
Note that a circular morphism is non-erasing, since otherwise
it would not be injective.

The following property, originally from
\cite{Restivo1974}, is well known. We will not use it but we
state it for the sake of clarity.
\begin{proposition}
A morphism is recognizable if and only if it is circular.
\end{proposition}
\begin{example}
  The Fibonacci morphism $\sigma\colon a\mapsto ab,b\mapsto a$
  is circular and thus recognizable.
\end{example}
\begin{example}
  The Thue-Morse morphism $\sigma\colon a\mapsto ab,b\mapsto ba$ is
  not circular. It is not recognizable since $(ab)^\infty$
  can be obtained as $\sigma(a^\infty)$ and as $S(\sigma(b^\infty))$.
  However, it is recognizable for aperiodic points since
  any sequence containing $aa$ or $bb$ has at most one factorization
  in $\{ab,ba\}$.
\end{example}
\begin{example}\label{example3.4}
  The morphism $\sigma\colon a\to aa,b\mapsto ab,c\mapsto ba$
  is not recognizable for aperiodic points. Indeed,
  every sequence without occurrence of $bb$ has
  two factorizations in words of $\{aa,ab,ba\}$.
  \end{example}

\begin{proposition}\label{propositionRecClosed}
  The family of morphisms recognizable
  for aperiodic points is closed under composition.
  \end{proposition}
\begin{proof}
  Let $\sigma=\alpha\circ\beta$ with $\beta\colon A^*\to B^*$ and
  $\alpha\colon B^*\to C^*$.
  Assume that $\alpha,\beta$ are recognizable for aperiodic points.
  Let $z\in C^\Z$ be an aperiodic point.
  Let $(x,k)$ be a centered $\sigma$-representation
  of $z$. Since $0\le k<|\sigma(x_0)|$, there is a decomposition
  $\beta(x_0)=ubv$ with $u,v\in B^*$ and $b\in B$ such that
  $|\alpha(u)|\le k<|\alpha(ub)|$. Set $\ell=|u|$,
  $m=|\alpha(u)|$ and $y=S^{\ell}(\beta(x))$.
  Then $y_0=b$ and $z=S^{k-m}(\alpha(y))$ (see Figure~\ref{figureComposition}).
  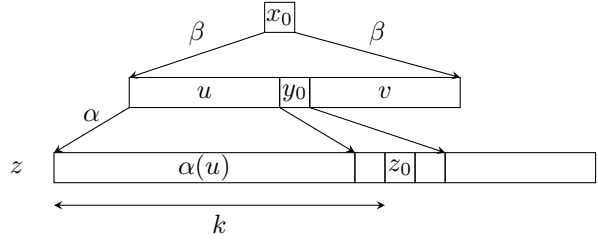
\begin{figure}[hbt]
    \centering
\tikzset{node/.style={draw,minimum size=0.4cm,inner sep=0pt}}
	\tikzset{title/.style={minimum size=0.5cm,inner sep=0pt}}
    \begin{tikzpicture}
      \node[node](z)at(2,2){$x_0$};
      \node[node,minimum width=2cm](u)at(1,1){$u$};
      \node[node](b)at(2.2,1){$y_0$};
      \node[node,minimum width=2cm](v)at(3.4,1){$v$};
      \node[title]at(-1.5,0){$z$};
      \node[node,minimum width=4cm](alphau)at(1,0){$\alpha(u)$};
      \node[node]at(3.2,0){};
      \node[node]at(3.6,0){$z_0$};
      \node[node]at(4,0){};
      \node[node,minimum width=2cm]at(5.2,0){};

      \draw[above,->,>=stealth](1.8,1.8)--node{$\beta$}(0,1.2);
      \draw[above,->,>=stealth](2.2,1.8)--node{$\beta$}(4.4,1.2);
      \draw[above,->,>=stealth](0,.8)--node{$\alpha$}(-1,.2);
      \draw[above,->,>=stealth](2,.8)--node{}(3,.2);
      \draw[above,->,>=stealth](2.4,.8)--node{}(4.2,.2);
      \draw[below,<->,>=stealth](-1,-.5)--node{$k$}(3.4,-.5);
    \end{tikzpicture}
    \caption{The composition $\sigma=\alpha\circ\beta$.}
    \label{figureComposition}
    \end{figure}
  Thus $(y,k-m)$ is a centered $\alpha$-representation of $z$ and
  $(x,\ell)$ is a $\beta$-representation of $y$.
  Conversely, if $(y,n)$ is an $\alpha$-representation of $z$
  and $(x,\ell)$ is a centered $\beta$-representation of $y$,
  then $(x,k)$ with $k=|\beta(y_{-\ell}\cdots y_{-1})|+n$ is
  a centered $\sigma$-representation of $z$. This shows the uniqueness
  of $(x,k)$.
\end{proof}

\section{Elementary morphisms}\label{sectionElementary}
\begin{definition}
A morphism $\sigma\colon A^*\to C^*$ is \emph{elementary}
(or \emph{indecomposable}) if
for every alphabet $B$ and
every pair of morphisms $\alpha\colon B^*\to C^*$ and $\beta\colon A^*\to B^*$
such that $\sigma=\alpha\circ\beta$, one has
$\Card(B)\ge \Card(A)$.
\end{definition}
If $\sigma\colon A^*\to C^*$ is elementary,  one has in
particular $\Card(C)\ge\Card(A)$ and moreover $\sigma$ is non-erasing.

\begin{example}\label{exampleTMelementary}
  The Thue-Morse morphism $\sigma\colon a\mapsto ab,b\mapsto ba$
  is elementary. Indeed, if $\sigma=\alpha\circ\beta$ with
  $\beta\colon \{a,b\}^*\to c^*$, then $ab=\alpha(c^i)$
  and $ba=\alpha(c^j)$ which is impossible.
  \end{example}

The notion of elementary morphism appears for the first time
in \cite{EhrenfeuchtRozenberg1978}.
A sufficient condition for a morphism to be elementary can
be formulated in terms of its incidence matrix.
 \begin{proposition}If
the rank of $M(\sigma)$ is
equal to $\Card(A)$, then
$\sigma$ is elementary.
 \end{proposition}
 \begin{proof} Indeed, if $\sigma=\alpha\circ\beta$
with $\beta\colon A^*\to B^*$ and $\alpha\colon B^*\to C^*$, then
\begin{displaymath}
  M(\sigma)=M(\beta)M(\alpha).
\end{displaymath}
If $\rank(M(\sigma))=\Card(A)$, then
$\Card(A)=\rank(M(\sigma))\le\rank(M(\alpha))\le\Card(B)$.
Thus $\sigma$ is elementary.
\end{proof}

This condition is not necessary.
For example, the Thue-Morse morphism $\sigma\colon a\mapsto ab,b\mapsto ba$
is elementary but its  incidence
matrix
\begin{displaymath}
  M(\sigma)=\begin{bmatrix}1&1\\1&1\end{bmatrix}
\end{displaymath}
 has rank one.

If $\sigma\colon A^*\to C^*$ is a morphism, we define
$\ell(\sigma)=\sum_{a\in A}(|\sigma(a)|-1)$. We say that
a decomposition $\sigma=\alpha\circ\beta$ with
$\alpha\colon B^*\to C^*$ and $\beta\colon A^*\to B^*$
is \emph{trim}
if
 \begin{enumerate}
 \item[\rm(i)] $\alpha$ is non-erasing,
   \item[\rm(ii)] for each
     $b\in B$ there is an $a\in A$ such that $\beta(a)$ contains $b$.
 \end{enumerate}
\begin{proposition}\label{propositionall(sigma)}
  Let $\sigma=\alpha\circ\beta$ with
  $\alpha\colon B^*\to C^*$ and $\beta\colon A^*\to B^*$ be a trim
  decomposition of $\sigma$.
 Then
\begin{equation}
  \ell(\alpha\circ\beta)\ge \ell(\alpha)+\ell(\beta).\label{eqIneqComposition}
\end{equation}
\end{proposition}
\begin{proof}
Set $\sigma=\alpha\circ\beta$. We have
\begin{eqnarray*}
  \ell(\sigma)-\ell(\beta)&=&\sum_{a\in A}(|\sigma(a)|-|\beta(a)|)
  =\sum_{a\in A}\sum_{b\in B}(|\alpha(b)||\beta(a)|_b-|\beta(a)|_b)\\
  &=&\sum_{a\in A}\sum_{b\in B}(|\alpha(b)|-1)||\beta(a)|_b
  =\sum_{b\in B}((|\alpha(b)|-1)\sum_{a\in A}|\beta(a)|_b).
\end{eqnarray*}
Since $\alpha$ is non-erasing, every factor $|\alpha(b)|-1$
is nonnegative.
Since every $b$ appears in some $\beta(a)$, every factor
$\sum_{a\in A}|\beta(a)|_b$ is positive, whence the conclusion.
\end{proof}

The following result is from~\cite{KarhumakiManuchPlandowski2003}.
It also appears in \cite{BertheSteinerThuswaldnerYassawi2019}
with the stronger hypothesis that
$\sigma\colon A^*\to B^*$ is such that $M(\sigma)$ has
rank $\Card(A)$.
\begin{theorem}\label{theoremIndecomposable}
  An elementary morphism 
  is recognizable  for aperiodic points.
\end{theorem}
We first prove Theorem~\ref{theoremIndecomposable} in a particular case.
A morphism $\sigma\colon A^*\to B^*$ with no erasable letter is \emph{left marked} if
every word $\sigma(a)$ for $a\in A$ begins with a distinct letter.
Symmetrically, it is \emph{right marked} if every word $\sigma(a)$
for $a\in A$ ends with a distinct letter.
In particular, if $\sigma$ is left marked,
$\sigma$ is injective on $A$ and $\sigma(A)$
is a prefix code. It is clear that a marked morphism is elementary.

The following statement, which is a particular case of Theorem~\ref{theoremIndecomposable} appears in \cite{BertheSteinerThuswaldnerYassawi2019}
(Lemma 3.3). We give an independent proof.
\begin{proposition}\label{lemmaLeftMarked}
If $\sigma\colon A^*\to B^*$ is left marked, then it is recognizable for
aperiodic points.
\end{proposition}

We fist show the following elementary lemma.
\begin{lemma}\label{lemmaLeftDelay}
  Let $\sigma\colon A^*\to B^*$ be an injective morphism and let $U=\sigma(A)$.
  If $\sigma$
  is not injective on $A^{\N}$, there exist words $u,v,w$ such that
  \begin{equation}
    u,uv,vw,wv\in U^*,\mbox{ and }v\notin U^*.\label{equvw}
  \end{equation}
\end{lemma}
\begin{proof}
  By the hypothesis, there exist $u_0,u_{1},\ldots$ and $u'_0,u'_{1},\ldots$
  with $u_i,u'_i\in U$ such that $u_0u_{1}\cdots =u'_0u'_{1}\cdots $.
  We may assume that $u_0\ne u'_0$. For every $n\ge 0$, there
  is $v_n$ and $n'\ge 0$ such that $u_0\cdots u_n = u'_0\cdots u'_{n'}v_n$.
  Since $\sigma$ is injective, we have $v_n\notin U^*$.
  We may moreover assume that $v_n$ is a prefix of $u'_{n'+1}$.
  Let $n<m$ be such that $v_n=v_m$. Set $u'_{n'+1}\cdots u'_{m'}=v_nw$
  (see Figure~\ref{figureuvw}). Then $u=u'_0\cdots u'_{n'}$,
  $v=v_n=v_m$ and $w$ satisfy \eqref{equvw}.
\end{proof}
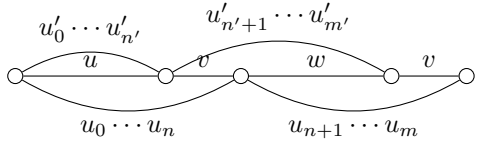
\begin{figure}[hbt]
\centering
\tikzset{node/.style={circle,draw,minimum size=0.2cm,inner sep=0pt}}
\tikzset{title/.style={minimum size=0.5cm,inner sep=0pt}}

\begin{tikzpicture}
  \node[node](0)at(0,0){};\node[node](1)at(2,0){};
  \node[node](2)at(3,0){};\node[node](3)at(5,0){};
  \node[node](4)at(6,0){};

  \draw[above](0)edge node{$u$}(1);
  \draw[bend left,above](0)edge node{$u'_0\cdots u'_{n'}$}(1);
  \draw[above](1)edge node{$v$}(2);
  \draw[bend right,below](0)edge node{$u_0\cdots u_{n}$}(2);
  \draw[bend left,above](1)edge node{$u'_{n'+1}\cdots u'_{m'}$}(3);
  \draw[above](2)edge node{$w$}(3);
  \draw[bend right,below](2)edge node{$u_{n+1}\cdots u_{m}$}(4);
  \draw[above](3)edge node{$v$}(4);
  \end{tikzpicture}
\caption{The words $u,v,w$.}\label{figureuvw}
  \end{figure}
\begin{proofof}{of Proposition~\ref{lemmaLeftMarked}}
  Set $U=\sigma(A)$. Since $U$ is a prefix code,
  $\sigma$ is injective on $A^\N$.
  If it is not injective on $A^{-\N}$, by the symmetric version
  of Lemma~\ref{lemmaLeftDelay}, there are words $u,v,w$
  such that $u,vu,vw,wv\in U^*$ but $v\notin U^*$.
  We may assume that $u,v,w$ are chosen of minimal length.
  Since $\sigma$ is left marked, either $u$ is a prefix of $w$
  or $w$ is a prefix of $u$. In the first case, set $w=uw'$.
  Then since $U$ is prefix, $u,wv=uw'v\in U^*$ imply $w'v\in U^*$.
  But then $(vu)(w'v)(u)=(vuw')(vu)$, a contradiction.
  Finally, if $u=wu'$, then $vw,vu=vwu'\in U^*$ imply $u'\in U^*$,
  and we may replace $u,v,w$ by $u',w,v$, a contradiction again.
  Thus $\sigma$ is injective on $A^{-\N}$.

  Assume now that $y\in B^\Z$ has two distinct centered $\sigma$-representations
  $(x,k)$ and $(x',k')$. We may assume $k=0$. Since
  $\sigma$ is injective on $A^{-\N}$, we have $k'\ne 0$.
  We will prove that $y$ is periodic.
  
  Let $P$ be the set of proper prefixes
  of $U$. For $p\in P$ and $a\in A$, there is at most one $q\in P$
  such that $p\sigma(a)\in U^*q$. We denote $q=p\cdot a$.
  Let $p_0=y_{-k'}\cdots y_{-1}$. Since $y=\sigma(x)=S^{k'}(\sigma(x'))$,
  we have
  \begin{eqnarray*}
    \sigma(\cdots x'_{-2}x'_{-1})p_0&=&\sigma(\cdots x_{-1})\\
p_0\sigma(x_0x_1\cdots)&=&\sigma(x'_0x'_1\cdots)
  \end{eqnarray*}
  As a consequence, there exists for each $n\in\Z$ a word $p_n\in P$
  such that $p_n\cdot x_n=p_{n+1}$.  Since $\sigma$ is left
  marked, there is for every nonempty $p\in P$ at most
  one $a\in A$ such that $p\cdot a$ is in $P$. But all $p_n$
  are nonempty. This is clear if $n<0$ since otherwise $p_0$
  is also empty. For $n>0$, we have 
  \begin{eqnarray*}
    \sigma(\cdots x'_{i_n})p_n&=&\sigma(\cdots x_n),\\
p_n\sigma(x_{n+1}\cdots)&=&\sigma(x'_{i_n+1}\cdots).
  \end{eqnarray*}
  and thus, since $\sigma$ is
  injective on $A^\Z$, $p_n=\varepsilon$ implies that $x=x'$.
  Thus $x$ is periodic and $y$ is also periodic.

  Consider the labeled graph with $P$ as set of vertices
  and edges $(p,a,q)$ if $p\cdot a=q$.
  By what we have seen,  the path $\cdots p_n\edge{x_n}p_{n+1}\cdots$
  is a cycle and thus  $y$ is periodic.
\end{proofof}
The graph used in the last part of the proof will appear
again in Section~\ref{section.efficient}.
\begin{example}
  The Thue-Morse morphism $\sigma\colon a\to ab,b\to ba$
  is left marked. Thus it is recognizable for aperiodic points (as
  we have already seen). The graph used in the proof of
  Proposition~\ref{lemmaLeftMarked} is represented
  in Figure~\ref{figureGraphMorse}.
  \begin{figure}[hbt]
\centering
\tikzset{node/.style={circle,draw,minimum size=0.4cm,inner sep=0pt}}
\tikzset{title/.style={minimum size=0.5cm,inner sep=0pt}}
\tikzstyle{every loop}=[->,shorten >=1pt,looseness=8]
\tikzstyle{loop above}=[in=50,out=130,loop]
\tikzstyle{loop below}=[in=-50,out=-130,loop]
\begin{tikzpicture}
  \node[node](0)at(0,1){$\varepsilon$};
  \node[node](a)at(2,1.5){$a$};\node[node](b)at(2,.5){$b$};

  \draw[below,->](0)edge[loop below]node{$b$}(0);
  \draw[above,->](0)edge[loop above]node{$a$}(0);
  \draw[above,->](a)edge[loop above]node{$b$}(a);
  \draw[below,->](b)edge[loop below]node{$a$}(b);
  \end{tikzpicture}
\caption{The graph associated with the Thue-Morse morphism.}\label{figureGraphMorse}
  \end{figure}
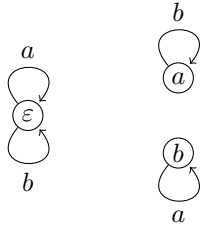
  \end{example}
The proof of Theorem~\ref{theoremIndecomposable} also uses the following
statement, originally due to \cite{Linna1977}
(see also \cite{BerstelPerrinPerrotRestivo1979} and \cite[Chapter
6]{Lothaire2002}).

\begin{proposition}\label{lemmaLinna}
  If a morphism $\sigma\colon A^*\to C^*$ is not injective on $A^\N$,
  there is a trim decomposition $\sigma=\alpha\circ\beta$ with
  $\alpha\colon B^*\to C^*$ and $\beta\colon A^*\to B^*$ such that
  $\alpha$ is injective on $B^\N$, $\Card(B)<\Card(A)$ and 
  every  $b\in B$ appears as the first letter of $\beta(a)$
  for some $a\in A$.
\end{proposition}
\begin{proof}
  Assume first that $\sigma$ is non-erasing.
  We use an induction on $\ell(\sigma)$. If $\ell(\sigma)=0$, set
  $B=\sigma(A)$. Let $\alpha$ be the identity on $B$ and let $\beta=\sigma$.
  All conditions are clearly satisfied.

  Assume now the statement true
  for $\ell<\ell(\sigma)$. Since $\sigma$ is not injective on $A^\N$,
  we have $\sigma(a_0a_1\cdots)=\sigma(a'_0a'_1\cdots)$ for some
  $a_i,a'_i\in A$ with $a_0\ne a'_0$.
  We can assume  that $\sigma(a_0)$ is a prefix of
  $\sigma(a'_0)$. Set $\sigma(a'_0)=\sigma(a_0)v$. If $v$ is empty,
  set $B=A\setminus \{a_0\}$. Let $\alpha$ be the restriction of $\sigma$ to
  $B$ and let $\beta$ be defined by $\beta(a'_0)=a_0$ and $\beta(a)=a$
  for $a\ne a'_0$. Clearly $\sigma=\alpha\circ\beta$ and all
  conditions are satisfied.

  Next, assume the $v$ is nonempty. Define $\alpha_1:A^*\to C^*$
  by $\alpha_1(a'_0)=v$ and $\alpha_1(a)=\sigma(a)$ for $a\ne a'_0$.
  Next, define $\beta_1:A^*\to A^*$ by $\beta_1(a'_0)=a_0a'_0$ and $\beta_1(a)=a$
  for $a\ne a'_0$. Then $\sigma=\alpha_1\circ\beta_1$ since
  \begin{displaymath}
    \alpha_1\circ\beta_1(a'_0)=\alpha_1(a_0a'_0)=\sigma(a_0)v=\sigma(a'_0)
    \end{displaymath}
  and $\alpha_1\circ\beta_1(a)=\alpha_1(a)=\sigma(a)$ if $a\ne a'_0$.
  The morphism $\beta_1$ is injective on $A^\N$ because no
  word in $\beta_1(A)$ begins with $a'_0$. Thus $\alpha_1$ is not
  injective on $A^\N$. By Equation \eqref{eqIneqComposition},
  we have $\ell(\alpha_1)<\ell(\sigma)$. By induction hypothesis,
  we have a decomposition $\alpha_1=\alpha_2\circ \beta_2$ for
  $\beta_2:A^*\to B^*$ and $\alpha_2:B^*\to C^*$
  with $\Card(B)<\Card(A)$, the morphism $\alpha_2$ being injective
  on $B^\N$ and every letter $b\in B$ appearing as initial
  letter in the word $\beta_2(a)$ for some $a\in A$.
  Note that, since $\alpha_1$ is non-erasing, $\beta_2$
  is non-erasing.
  
  Set $\beta=\beta_2\circ\beta_1$. Since $\beta_1,\beta_2$ are
  non-erasing, $\beta$ is non-erasing.
  Then $\sigma=\alpha_1\circ\beta_1=\alpha_2\circ\beta_2\circ\beta_1=\alpha_2\circ\beta$. The decomposition $\sigma=\alpha_2\circ\beta$
  satisfies all the required conditions. Indeed, let $b \in B$.
  Then there is $a \in  A$ such that $b$ is the first letter of $\beta_2(a)$.
  If $a \ne a_0$, we have $\beta_1(a) = a$ and thus $b$
  is the first letter of $\beta(a)$.
  Suppose next that $a=a'_0$. Since
  $\sigma(a_0a_1\cdots)=\sigma(a'_0a'_1\cdots)$ and since $\alpha_2$
  is injective on $B^\N$, we have
  $\beta(a_0a_1\cdots)=\beta(a'_0a'_1\cdots)$.
  Since $\beta_1(a_0)=a_0$ and
  $\beta_1(a'_0)=a_0a'_0$, we obtain
  $\beta_2(a_0)\beta(a_1\cdots)=\beta_2(a_0a'_0)\beta(a'_1\cdots)$ and thus
  \begin{displaymath}
    \beta(a_1\cdots)=\beta_2(a'_0)\beta(a'_1\cdots),
  \end{displaymath}
  showing, since $\beta$ is non-erasing,
  that $b$ is the initial letter of $\beta(a_1)$.

  Consider now a morphism $\sigma$
  such that the set $B=\{a\in A\mid \sigma(a)\ne\varepsilon\}$
  is strictly contained in $A$. Let $\beta\colon A^*\to B^*$
  be defined by $\beta(a)=a$ if $a\in B$ and $\beta(a)=\varepsilon$
  otherwise. Let $\alpha$ be the restriction of $\sigma$ to $B^*$.
  Then $\sigma=\alpha\circ\beta$ and $\alpha$ is non-erasing.
  If $\alpha$ is injective on $B^\N$, we are done.
  Otherwise, by the first part of the proof,
  we have $\alpha=\alpha_1\circ\beta_1$ with $\alpha_1\colon B_1^*\to C^*$
  and $\beta_1\colon B^*\to B_1^*$ with $\alpha_1$ injective on
  $B_1^\N$, $\Card(B_1)<\Card(B)$ and every $b_1\in B_1$
  appears as the first letter of some $\beta_1(b)$.
  Then the decomposition $\sigma=\alpha_1\circ(\beta_1\circ\beta)$
  satisfies all the conditions.
\end{proof}

\begin{example}
  Set $A=C=\{a,b,c\}$.
  The morphism $\sigma\colon a\mapsto ab$, $b\mapsto abc$, $c\mapsto cc$
  is not injective
  on $A^\N$ because $\sigma(ac^\omega)=\sigma(bc^\omega)=abc^\omega$.
  The decomposition $\sigma=\alpha\circ\beta$ with
  $\alpha\colon u\mapsto ab,v\mapsto c$ and $\beta\colon a\mapsto u,b\mapsto uv,c\mapsto vv$
  satisfies the conditions of Proposition~\ref{lemmaLinna}.
\end{example}
By a symmetric version of Proposition~\ref{lemmaLinna}, an elementary morphism
$\sigma\colon A^*\to A^*$ is injective on $A^{-\N}$. Since
a morphism which is injective on $A^\N$ and on $A^{-\N}$
is injective on $A^\Z$, we obtain the following corollary of Proposition~\ref{lemmaLinna}.
\begin{corollary}\label{propositionInjectiveZ}
  An elementary morphism $\sigma\colon A^*\to C^*$ is injective
  on $A^\Z$.
  \end{corollary}

\begin{proofof}{of Theorem~\ref{theoremIndecomposable}}
  Let $\sigma\colon A^*\to B^*$ be an elementary morphism.
We use an induction on $\ell(\sigma)$.
Since $\sigma$ is elementary, it has no erasable letter
and the minimal possible value of $\ell(\sigma)$ is
$0$. In this case, $\sigma$ is a bijection
from $A$ to $B$ and thus it is recognizable.

Assume now that $\sigma$ is not recognizable on aperiodic points.
Thus there exist $x,x'\in A^\N$ and $w$ with $0<|w|<|\sigma(x_0)|$
such that $\sigma(x)=w\sigma(x')$ for some proper suffix $w$ of $\sigma(a')$.
Set $\sigma(a')=vw$. We can then write $\sigma=\sigma_1\circ\tau_1$
with $\tau_1\colon A^*\to A_1^*$ and $\sigma_1\colon A_1^*\to B^*$ and $A_1=A\cup\{a''\}$ where $a''$ is a new letter. We have
$\tau_1(a')=a'a''$ and $\tau_1(a)=a$ otherwise. Next
$\sigma_1(a')=v$, $\sigma_1(a'')=w$ and $\sigma_1(a)=\sigma(a)$ otherwise.
Since $\ell(\tau_1)>0$, we have  $\ell(\sigma_1)<\ell(\sigma)$
by Equation~\eqref{eqIneqComposition}.

Since $\tau_1$ is injective on $A^\N$,
$\sigma_1$ is not injective on $A_1^\N$.
By Proposition~\ref{lemmaLinna}, we can write $\sigma_1=\sigma_2\circ\tau_2$
with $\sigma_2:A_2^*\to B^*$ and $\tau_2:A_1^*\to A_2^*$
for some alphabet $A_2$ such that $\Card(A_2)<\Card(A_1)$
and  that every letter $c\in A_2$
appears as the first letter of some $\tau_2(a)$ for $a\in A_1$.
Then, by Equation~\eqref{eqIneqComposition}, we have
\begin{equation}
  \ell(\sigma_1)\ge \ell(\sigma_2)+\ell(\tau_2).\label{ineq2}
  \end{equation}

If $\Card(A_2)<\Card(A)$, then $\sigma$ is not elementary and we are done.
Otherwise, we have $\Card(A_2)=\Card(A)$. We may also assume
that $\sigma_2$ and $\tau_2\circ \tau_1$ are elementary
since otherwise $\sigma$ is not elementary.

Since $\sigma_2$ is elementary and since $\ell(\sigma_2)\le\ell(\sigma_1)<\ell(\sigma)$, by the induction hypothesis, $\sigma_2$ is recognizable for
aperiodic points.

Since $\sigma=\sigma_2\circ\tau_2\circ\tau_1$, we have also
\begin{equation}
  \ell(\sigma)\ge \ell(\sigma_2)+\ell(\tau_2\circ\tau_1).\label{ineq3}
  \end{equation}

Thus, if $\ell(\sigma_2)>0$,  the
inequality $\ell(\tau_2\circ\tau_1)<\ell(\sigma)$ holds. Since
$\tau_2\circ\tau_1$ is elementary, we obtain that
$\tau_2\circ\tau_1$ is recognizable for aperiodic points by
induction hypothesis. Thus, by Proposition~\ref{propositionRecClosed}
that $\sigma$ is recognizable for aperiodic points.

Let us finally assume that $\ell(\sigma_2)=0$. Then
$\sigma$ is left marked because $\sigma_2$
is a bijection from $A_2$ onto $B$ and every letter
of $A_2$ appears as initial of $\tau_2(a)$ for $a\in A_1$.
We obtain the conclusion by
Proposition~\ref{lemmaLeftMarked}.
\end{proofof}
\section{Recognizability and iterated morphisms}\label{sectionIterated}
In this section, we address the problem of the recognizability
of morphisms $\sigma\colon A^*\to A^*$ on the shift $X(\sigma)$.

The following is proved in~\cite[Proposition
5.10]{BezuglyiKwiatkowskiMedynets2009} for non-erasing morphisms.
We give a proof which holds for morphisms with erasable letters.
\begin{theorem}\label{propositionBKM}
  Let $\sigma\colon A^*\to A^*$ be a morphism. Every point $y$ in $X(\sigma)$
  has a $\sigma$-representation $y=S^k(\sigma(x))$ with $x\in X(\sigma)$.
\end{theorem}
\begin{proof}
Let $k=|\sigma|$.
Let $y$ be in $X(\sigma)$. For every $n > 2k$, there is an integer
$m \geq 1$ such
that $y_{[-n,n]}$ is a factor of $\sigma^m(a)$ for some letter $a
\in A$.

Hence $y_{[-n, n]} = s \sigma(a_1 \ldots a_\ell) p$, where $a_i
\in A$, $a_1 \ldots a_\ell$ is a factor of $\sigma^{m-1}(a)$,
$s$ is a suffix of a word in $\sigma(A)$ and $p$ is a prefix of a word in
$\sigma(A)$. Since $n > 2k$, at leat one letter $a_i$ is such that
$\sigma(a_i) \neq \varepsilon$. We write $a_1 \ldots a_\ell = u_1
\ldots u_r$ where $r \geq 1$, $u_j = v_jb_j$ with $b_j \in A, v_j \in A^*$ (the word
$v_j$ being possibly the empty word), $\sigma(v_j)=\varepsilon$
and $\sigma(b_j) \neq \varepsilon$.

Then there are integers $n \leq k_1 < k_2 < \cdots
< k_{r+1} \leq  n+1$ such that
$y_{[-n, n]} = s \sigma(u_1) \ldots \sigma(u_r) p$,
with $s = y_{[-n, k_1)}$, $\sigma(u_j) = y_{[k_j,k_{j+1})}$, $p= y_{[k_{r+1},n]}$.
Note that $\sigma(u_j)=\sigma(a_j)$ is non-empty and belongs to
$\sigma(A)$.

For every $i$ with $-n + k \leq i \leq n-k$, we denote by $f_n(i)$
the unique triple $(k_j, k_{j+1}-1, u_j)$ defined above such that $k_j \leq i <
k_{j+1}$.

Since $\sigma(A)$ is finite and since there is a finite number of erasable
words in $\cL(\sigma)$ by Proposition~\ref{propositionErasing},
there is a finite number of  triples $f_n(0)$. Thus there is a
strictly increasing sequence $(i_n)_{n \in
\N}$ of natural integers such that $f_{i_n}(0)$ is constant.
We denote this constant value by $f(0)$.

Taking a 
subsequence of $(i_n)_{n \in \N}$ we may also assume that all $f_{i_n}(-1)$ and $f_{i_n}(1)$ are equal
to some triple $f(-1)$ and $f(1)$ respectively. By iterating his
process we define for each $i \in \Z$ a triple $f(i)=(j_i, \ell_i,
u_i)$.
By construction $f(i) = f(i')$ for all $i, i' \in [j_i, \ell_i)$.

We now concatenate the words $u_i$ in order to build a point $x$ as follows.
We define a strictly increasing sequence $(j_n)_{n \in \Z}$ of integers by
$j_0 = 0$, for $n > 0$, $j_n$ is the least integer larger than
$j_{n-1}$ such that $f(j_n) \neq f(j_{n-1})$, and for $n < 0$,
$j_n$ is the largest integer less than
$j_{n+1}$ such that $f(j_n) \neq f(j_{n+1})$.
Let $x= \ldots u_{j_{-2}}u_{j_{-1}}.u_{j_0}u_{j_1}u_{j_2}\ldots $.
By construction $x \in X(\sigma)$ and there is some integer $r$ such that $y = S^r(\sigma(x))$.
\end{proof}

We will also need the following more technical statements.
The first one concerns growing letters.
Let $\sigma\colon  A^* \rightarrow A^*$ be a morphism.
Recall that a letter
$a$ is
growing if $|\sigma^n(a)|$
is unbounded and that in this case, by Proposition~\ref{lemmaGrowing},
one has actually $\lim_{n \rightarrow +\infty} |\sigma(a)| = + \infty$.
\begin{proposition} \label{lemma.onesided}
Let $\sigma\colon  A^* \rightarrow A^*$ be a morphism and let $S$ be
the set of right-infinite sequences $x$ such that every $x_{[0, n]}$ is a prefix of
some $\sigma^m(a)$ for a fixed letter $a$.

If $x \in S$ contains a growing
letter, there are integers $r$
and $0 \leq k < r$ such that
$x = \lim_{n \rightarrow +\infty}\sigma^{rn+k}(a)$.

The set of points $x \in S$ with
no growing letters is a finite set of
eventually periodic right-infinite sequences.
\end{proposition}
\begin{proof}
  We first assume that $x$ contains a growing letter. Let $x_i= b$ be the
  first growing letter of $x$. Thus $x = x_{[0, i)}bx'$ where
  $x_{[0,i)}$ is non-growing. There are integers
  $m < m'$ such that $x_{[0,i)}b$ is a prefix of $\sigma^m(a)$ and
  of $\sigma^{m'}(a)$. Let $r = m'-m$.

  Since $x_{[0,i)}b$ is a prefix of $\sigma^m(a)$,
  $\sigma^{r}(x_{[0,i)}b)$ is a prefix of $\sigma^{m'}(a)$
  and thus $\sigma^{r}(x_{[0,i)}b)$ is prefix comparable
  with $x_{[0,i)}b$.
  Since $x_{[0,i)}$ is non-growing,
  $\sigma^{r}(x_{[0, i)})$ is a prefix $x_{[0, j)}$ of $x_{[0, i)}$
  and $\sigma^{r}(b) = x_{[j,i)}bu$ where $u$ is a non-erasable word
  (otherwise, $b$ would not be growing).

  Then for all integers $n$,
  $\sigma^{nr}(x_{[0, i)}b) = x_{[0,i)}bu\sigma^r(u) \cdots\sigma^{(n-1)r}(u)$.
      Thus $\lim_{n \rightarrow +\infty}  \sigma^{m+rn}(a)$ exists.
It follows that $\lim_{n \rightarrow +\infty}
  \sigma^{m+k+rn}(a)$ exists for every $0 \leq k < r$.
 It also follows that $x$ is one of the points $x^{(k,r)}= \lim_{n \rightarrow +\infty}
 \sigma^{k+rn}(a)$ for some $0 \leq k < r$.
 
Let us show that there is a finite number of such points. 
If $y$ is another point of $S$ with a growing letter, $y = \lim_{n \rightarrow +\infty}
\sigma^{k'+r'n}(a)$ with $r' > r$, $0 \leq k' < r'$, then we have
$y = x^{(k",r)}$ for some $0 \leq k" < r$.

We now assume that all letters of $x$ are non-growing.
There are nonnegative integers $r, p \leq \Card(A)$ such that $\sigma^r(a) = ubv$
and $\sigma^p(b) = wbz$ with $u, w$ non-growing, $b \in A$ growing.
We get that for any $k \geq 0$,
\begin{equation*}
  \sigma^{r + kp}(a) = \sigma^{kp}(u) \sigma^{(k-1)p}(w)\cdots \sigma^p(w)wbz^{(k)}.
\end{equation*}
Since $u$ and $w$ are non-growing there are nonnegative integers $i, i'$ such
that $\sigma^{i'p}(u) = \sigma^{i'p + ip}(u)$ and
$\sigma^{i'p}(w) = \sigma^{i'p + ip}(w)$.
It follows that
\begin{equation*}
  \sigma^{r + i'p + kip}(a) = \sigma^{i'p}(u) \sigma^{((i'+ki)-1)p}(w)\cdots \sigma^p(w)wbz^{(i'+ki)}.
\end{equation*}
Note that $w$ non-erasable since otherwise $x$ would have a growing
letter. Thus for large enough $k$ we get
\begin{equation*}
  \sigma^{r + i'p + kip}(a) = t z^{k-1}
\sigma^{(i'-1)p}(w)
  \cdots \sigma^p(w)wbz^{(i'+ki)}.
\end{equation*}
where $t = \sigma^{i'p}(u) \sigma^{(i' + i -1)p}(w) \cdots
\sigma^{i'p}(w)$, and $z = \sigma^{(i'+i -1)p}(w)   \cdots  \sigma^{i'p}(w)$.

Every $x_{[0, n]}$ is a prefix of some $\sigma^{N_n} (a)$.
There is an infinite number of $n$ such that $N_n$ is equal to $r +
i'p + k_nip + h$ for some fixed $0 \leq h < ip$ for some nonnegative integer $k_n$.
It follows that for an infinite number $n$, $x_{[0,n]}$ is a prefix of
$\sigma^h(tz^{k_n-1}t')$ where $t' = \sigma^{(i'-1)p}(w) \cdots \sigma^p(w)w$.
It follows that there is a finite number of words $u,v$
depending only on $\sigma$ such that $x = uv^\omega$.
Hence the conclusion.
\end{proof}


The second one gives a sufficient condition to have
$X(\sigma)=X(\sigma^n)$. This is not always true,
as shown by the example of $\sigma\colon a\mapsto bc,
b\mapsto cc,c\mapsto bb$. We have $^\omega c\cdot b^\omega$
is in $X(\sigma)$ but not in $X(\sigma^2)$.

\begin{lemma}\label{propositionX(sigman)=X(sigma)}
Let $\sigma\colon A^*\to A^*$ be a morphism. If every
letter $a\in A$ appears in $\sigma(A)$, then $X(\sigma^n)=X(\sigma)$
for every $n\ge 1$.
\end{lemma}
\begin{proof}
  It is enough to show that $\cL(\sigma^n)=\cL(\sigma)$.
  We first clearly have $\cL(\sigma^n)\subset \cL(\sigma)$.
  Conversely, let $u\in\cL(\sigma)$. Then $u$  is a factor of some
$\sigma^{m}(a)$. Let $k$ be such that $kn \leq m < (k+1)n$.
Since every letter appears in  $\sigma(A)$,
 the letter $a$ appears in $\sigma^{(k+1)n- m}(b)$ for
some letter $b$ and thus $u$ is a factor of $\sigma^{n(k+1)}(b)$.
Hence $u \in \cL(\sigma^n)$.
\end{proof}

A morphism $\sigma\colon A^*\to A^*$ is \emph{primitive} if there
is an integer $n\ge 1$ such that every $b\in A$ appears
in every $\sigma^n(a)$ for $a\in A$.

A morphism $\sigma\colon A^*\to A^*$ is \emph{aperiodic} if $X(\sigma)$
does not contain periodic points.

The following result is proved in \cite{BertheSteinerThuswaldnerYassawi2019}
in the case of non-erasing morphisms.
It is a generalization of Moss\'e's theorem asserting that
every aperiodic primitive morphism $\sigma$ is recognizable on
$X(\sigma)$. Our new proof holds for the general case of morphisms
with erasable letters
and relies on Theorem~\ref{theoremIndecomposable}. It also
gives a proof of the original version of Moss\'e's theorem which is easier than the previous ones
(see the comment after the proof).

\begin{theorem}\label{theoremBSTY}
Every morphism $\sigma\colon A^*\to A^*$ is recognizable on $X(\sigma)$ for aperiodic points.
\end{theorem}

  \begin{proof}
    Let $\sigma\colon A^*\to A^*$ be a morphism.
    
    \noindent{\bf Case 1.}
We first assume that every letter of $A$ is a letter of some word of
$\sigma(A)$. By Lemma~\ref{propositionX(sigman)=X(sigma)}
we have $X(\sigma) = X(\sigma^n)$ for every positive
integer $n$. 

It is enough
  to prove the statement for a power $\sigma^n$ of $\sigma$.
  Indeed, assume that $\sigma^n$ is recognizable on $X(\sigma^n)=X(\sigma)$
  for aperiodic points. Let $x$ be an aperiodic point $x\in X(\sigma)$
  and let $(z,K)$ be its unique centered $\sigma^n$-representation.
   Let $\sigma^{n-1}(z_0)=uav$ with $a\in A$ be the
  unique factorization of $\sigma^{n-1}(z_0)$ such that
  $|\sigma(u)|\le K<|\sigma(ua)|$. Set $y=S^{|u|}\sigma^{n-1}(z)$ and $\sigma(a)=sx_0t$.
  Then $(y,|s|)$ is the unique centered
  $\sigma$-representation  of $x$.
  
  Indeed, it is a centered $\sigma$-representation of $x$.
  If $(y',k')$ is another centered $\sigma$-representation of $x$,
  let $(z',r')$ be a centered $\sigma^{n-1}$-representation of $y$.
  Set $\sigma^{n-1}(z'_0)=u'a'v'$ with $|u'|=r'$ and
  $\sigma(a')=s'x_0t'$ with $|s'|=k'$.
  Then $u'a'$ is a prefix of $\sigma^{n-1}(z'_0)$
  such that $|\sigma(u')|\le |\sigma(u')|+k' <|\sigma(u'a')|$.
  Since $(z',|\sigma(u')|+k')$ is a centered $\sigma^n$-representation
  of $x$, we have $z = z'$, $K=|\sigma(u')|+k'$ and thus $ua=u'a'$,
  which implies $y=y'$ and $k=k'$.

  Choose $n$  such that  $\sigma^n$
  has a decomposition $\sigma^n=\alpha\circ\beta$
  with $\beta\colon A^*\to B^*$ and $\alpha\colon B^*\to A^*$
  and $\Card(B)$ minimal.
  Then  $\tau=\beta\circ\alpha$ is elementary. Indeed,
  if $\tau=\gamma\circ\delta$ with $\delta:B^*\to C^*$
  and $\gamma\colon C^*\to B^*$ is a decomposition of $\tau$
  such that $\Card(C)<\Card(B)$, then $\sigma^{2n}=(\alpha\circ\gamma)\circ(\delta\circ\beta)$ is a decomposition of $\sigma^{2n}$ with
  $\delta\circ\beta\colon A^*\to C^*$ and $\alpha\circ\gamma\colon C^*\to A^*$
  such that $\Card(C)<\Card(B)$, a contradiction.

  It is easy
  to check that $\alpha$  is also elementary.
Moreover, we have 
  \begin{equation}
    \beta(X(\sigma))\subset X(\tau).
    \label{eqSigmaTau}
  \end{equation}
 
    \begin{figure}[hbt]
    \centering
\tikzset{node/.style={draw,minimum size=0.4cm,inner sep=0pt}}
\tikzset{node1/.style={circle,draw,minimum size=0.1cm,inner sep=0pt}}

\begin{tikzpicture}
          \node[node]at(2,4){$b$};
          \node[node]at(2,3){$a$};
          \node[node1](w1)at(1,.5){};
          \node[node1](w2)at(3,.5){};
          \draw[->,left,>=stealth](1.8,3.8)--node{$\sigma^n$}(1,2.8);
          \draw[->,>=stealth](2.2,3.8)--node{}(3,2.8);
          \draw(1,2.8)--(1.8,2.8);\draw(2.2,2.8)--(3,2.8);
          \draw[->,left,>=stealth](1.8,2.8)--node{$\sigma^{Nn}$}(.5,1.5);
          \draw[->,>=stealth](2.2,2.8)--node{}(3.5,1.5);
          \draw(.5,1.5)--(3.5,1.5);
          \draw[->,left,>=stealth](.5,1.5)--node{$\beta$}(0,1);
          \draw[->,>=stealth](3.5,1.5)--node{}(4,1);
          \draw(0,1)--(4,1);
          \draw[below](w1)--node{$w$}(w2);
          \node[node]at(7,4){$b$};
          \node[node]at(7,3.5){$c$};
          \node[node1](w3)at(6,.5){};
          \node[node1](w4)at(8,.5){};
          
          \draw[->,left,>=stealth](6.8,3.8)--node{$\beta$}(6.5,3.3);
          \draw[->,>=stealth](7.2,3.8)--(7.5,3.3);
          \draw(6.5,3.3)--(7.5,3.3);
          \draw[->,left,>=stealth](6.8,3.3)--node{$\tau^{(N+1)n}$}(5,1);
          \draw(5,1)--(9,1);
          \draw[->,>=stealth](7.2,3.3)--node{}(9,1);
          \draw[below](w3)--node{$w$}(w4);
        \end{tikzpicture}
        \caption{proving that $w$ is a factor of $\tau^m(c)$.}
        \label{figureBeta(X)}
        \end{figure}
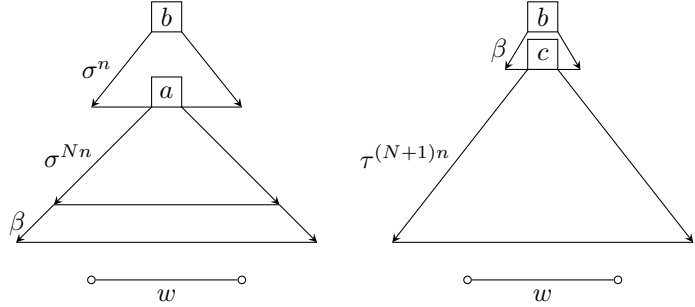
   Indeed, consider $x\in X(\sigma)$ and let $w$ be a factor
  of $\beta(x)$. We have to prove that $w$ is a factor
  of some $\tau^m(c)$ for $c\in B$. Now, since $x \in X(\sigma^n)$, there in an $N\ge 1$
  such that $w$ is a factor of some $\beta\circ\sigma^{Nn}(a)$
  for $a\in A$ and $N\ge 1$ and also a factor of
  $\beta\circ\sigma^{(N+1)n}(b)$ with $a$ factor of $\sigma^n(b)$
  (see Figure~\ref{figureBeta(X)}). Since $\sigma^n=\alpha\circ\beta$,
  there is a letter $c$ such that $a\in\alpha(c)$ and $c \in \beta(b)$.
  Since $\beta\circ\sigma^n=\tau^n\circ\beta$,
  we obtain that $w$ is a factor of $\tau^{(N+1)n}(c)$.
  
    Let $x\in X(\sigma)$ be an aperiodic point. Consider two centered
    $\sigma^n$-representa\-tions $(y,k)$ and $(y',k')$
    of $x$ with $y,y'\in X(\sigma)$. 
    Let $(z, \ell)$ and $(z', \ell')$
    be centered $\sigma^n$-representations  of $y$ and $y'$ respectively
    with $z,z'\in X(\sigma)$. Then $(\beta(z),m)$ and $(\beta(z'),m')$
    are,
    for some unique $m,m'$, some centered $\tau$-representations of $\beta(y)$
    and $\beta(y')$ respectively (see Figure~\ref{figurexyz}).
    The fact that $m,m'$ are unique results from the
    fact that $\tau$ is elementary and thus recognizable for aperiodic
    points by Theorem~\ref{theoremIndecomposable}.
    
    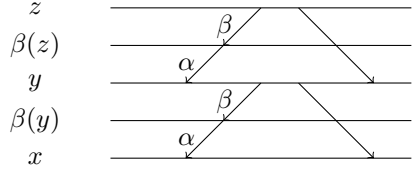
\begin{figure}[hbt]
    \centering
\tikzset{node/.style={draw,minimum size=0.4cm,inner sep=0pt}}
\tikzset{title/.style={minimum size=0.5cm,inner sep=0pt}}

\begin{tikzpicture}
  \node[title]at(0,2){$z$};\draw(1,2)--(5,2);
  \node[title]at(0,1.5){$\beta(z)$};\draw(1,1.5)--(5,1.5);
  \node[title]at(0,1){$y$};\draw(1,1)--(5,1);
  \draw[->,left](3,2)--node{$\beta$}(2.5,1.5);
  \draw[->,left](2.5,1.5)--node{$\alpha$}(2,1);
  \draw[->,left](3.5,2)--node{}(4.5,1);
  \node[title]at(0,.5){$\beta(y)$};\draw(1,.5)--(5,.5);
  \node[title]at(0,0){$x$};\draw(1,0)--(5,0);
  \draw[->,left](3,1)--node{$\beta$}(2.5,.5);
  \draw[->,left](2.5,.5)--node{$\alpha$}(2,0);
  \draw[->,left](3.5,1)--node{}(4.5,0);
\end{tikzpicture}
\caption{The points $x,y,z$.}\label{figurexyz}
\end{figure}
    Since $\alpha$ is elementary,
    it is recognizable for aperiodic points.
    Since $\alpha(\beta(y))=\sigma^n(y)$
    and $\alpha(\beta(y'))=\sigma^n(y')$, this implies that
    $\beta(y),\beta(y')$ belong to the same orbit.
    Since $\tau$ is recognizable for aperiodic points,
    we obtain that $\beta(z)$ and $\beta(z')$ belong to the same
    orbit. Applying $\alpha$ again, we conclude that
    $y,y'$ belong to the same orbit.
Since $x$ is aperiodic, this implies that $y = y'$ and $k = k'$. Thus $\sigma^n$
    is recognizable for aperiodic points.

 \noindent{\bf Case 2.}      We now prove the general case. We prove the result by induction on
   the number   of letters which do not appear in some word of
   $\sigma(A)$.
   The case where there is no such letter is Case 1 and thus the property holds. Let $a$ be a letter of $A$ which does
   not appear in a word in $\sigma(A)$ and set $B= A \setminus \{a\}$. We denote
   by $\sigma_B$ the morphism $\sigma$ restricted to $B^*$.
   By induction hypothesis,  $\sigma_B$ is recognizable on $X(\sigma_B)$
   for aperiodic points. 
   We denote by $K$ the maximal length of all $\sigma^n(b)$ for
   non-growing letters $b$ in $B$.

  Let $y$ be an aperiodic point having two centered $\sigma$-representations
  $(x, k)$ and $(x', k')$ with $x, x' \in X(\sigma)$. We
  cannot have $x,x'\in X(\sigma_B)$ and we may assume that
  $x$ belongs to $X(\sigma) \setminus X(\sigma_B)$. Thus there is an
  integer $m$ such that $x_{[-m, m]}$ is not a factor of any $\sigma^n(b)$ for
  $b \in B$. For each integer $n \geq m$, there is an
  integer $N(n)$ such that $x_{[-n,n]}$ is a factor of 
  $\sigma^{N(n)}(a)$.
  
  We have $\sigma(a) = b_1 \ldots b_\ell$ with $b_i \in B$.
  Since $x_{[-m, m]}$ is not a factor of any $\sigma^n(b_i)$, there are
  integers $1 \leq i_n \leq \ell$, $-m \leq k_n <  m$ such that $x_{[-n, k_n)}$
  is a suffix de $\sigma^{N(n)-1}(b_1 \ldots b_{i_n})$ and $x_{[k_n,
  n]}$ is a prefix  of $\sigma^{N(n)-1}(b_{i_n+1} \ldots b_\ell)$.
  Let $b_{j_n}$ be the last growing letter of $b_1 \ldots b_{i_n}$ and
  $b_{r_n}$ be the first growing letter of $b_{i_n+1} \ldots b_\ell$.
  Note that such growing letters exist for large enough $n$.

  There is an infinite number of integers $n$ such that $i_n = i$,
  $j_n = j$, $r_n = r$, $k_n = M$, where $1 \leq j \leq i < r \leq  \ell$
  and $-m \leq M < m$.

  Since $b_{j+1} \ldots b_i$ and $b_{i+1} \ldots b_{r-1}$ are
  non-growing, the lengths of all words $\sigma^{N(n)-1}(b_{j+1} \ldots
  b_i)$ and $\sigma^{N(n)-1}(b_{i+1} \ldots b_{r-1})$ are bounded by
  $K \ell$. Thus there is an infinite number of integers $n$ such that
  moreover $|\sigma^{N(n)-1}(b_{j+1} \ldots b_i)| = p$,
  $|\sigma^{N(n)-1}(b_{i+1} \ldots b_{r-1})| = p'$
  for some fixed integers $0 \leq p, p' \leq K\ell$.

  Let $b = b_j$ and $c = b_r$. Since $b, c$ are growing, for each $n$
  there are  integers $m, m'$
  such that $x_{[-n, M-p)}$ is a suffix of $\sigma^m(b)$ and $x_{[M+p',n]}$
  is a prefix of $\sigma^{m'}(c)$.

 By Proposition \ref{lemma.onesided}, there are finite sets $L, R$
of  respectively left and right-infinite sequences, and a
finite set $U$ of words of
length at most $2K\ell$
such that $x$ is a shift of $tut'$ with $t \in L$, $u \in U$ and $t'
\in R$. Thus there is a finite number of such orbits.

 We define a sequence $(x^{(i)})_{i \geq 0}$ of two-sided infinite sequences $x^{(i)}$
  by $x^{(0)} = x$
  and $x^{(i)} \in X(\sigma) \setminus X(\sigma_B)$ such that
  $x^{(i-1)} =S^k\sigma(x^{(i)})$ for
some $k$. There is an infinite number of $i$ such that $x^{(i)}$ are
shifts of the same sequence. Thus there are integers $i, j$ with $j \geq
1$ such that
$\sigma^j(x^{(i)})$ is a shift of  $x^{(i)}$.  As a consequence $y \in
X(\sigma) \setminus X(\sigma_B)$. Thus $x, x' \in  X(\sigma) \setminus
X(\sigma_B)$.

As above, we define sequences $(x^{(i)})$ and $({x'}^{(i)})$ for $x$ and $x'$. Thus
there are integers $i, j$ with $j \geq 1$ such that
$\sigma^j(x^{(i)})$ is a shift of  $x^{(i)}$ and 
$\sigma^{j}({x'}^{(i)})$ is a shift of ${x'}^{(i)}$. We get that there is
an integer $i$ such that $x^{(i)}$ and ${x'}^{(i)}$ are in the
same orbit and thus $x$ and $x'$ also.

Now if $x \neq x'$ of $k \neq k'$ with $0 \leq
k < |\sigma(x_0)|$ and $0 \leq k' < |\sigma(x'_0)|$,  we
get that $y$ is periodic, which is excluded.
Thus $y$ has a unique centered $\sigma$-representation
in $X(\sigma)$. 
\end{proof}
  Note that, to prove  Moss\'e's Theorem itself, one only needs the
  first case in the above proof. Indeed, if $\sigma$ is primitive,
  every letter appears in $\sigma(A)$. Note also that in this case
  Proposition \ref{lemma.onesided} is not used.
  The previous proofs of Moss\'e's Theorem work on a fixed point
  of the morphism and handle indices of this sequence, resulting
  in fairly complicated arguments (see for example~\cite{Kurka2003}).
  The proof given in~\cite{DurandPerrin2021} avoids these
  indices but remains essentially the same as that of \cite{Kurka2003}.
  The use of the notion of elementary morphism,
  through Theorem~\ref{theoremIndecomposable}, seems to us
  a major simplification.
  
We give below an example of a morphism $\sigma$ such that
$X(\sigma)$ contains both periodic and non-periodic points.
\begin{example}
  Let $\sigma$ be the morphism
  $\sigma\colon a\mapsto bac,b\mapsto bb,c\mapsto cd,d\mapsto c$. The
  set $X(\sigma)$ contains the aperiodic point $^\omega b\cdot a\sigma^\omega(c)$
  where $\sigma^\omega(c)=cdccd\cdots$ is the one-sided Fibonacci sequence.
  It also contains the periodic point $y=b^\infty$.
  Every point, except $y$, has a unique centered $\sigma$-representation.
  \end{example}
\section{Automata and syntactic groups} \label{sectionSyntactic}
We now introduce notions from automata theory which will
allow us to formulate a characterization of
morphisms $\sigma\colon A^*\to B^*$ which are
recognizable for aperiodic points (Theorem~\ref{theoremSyntactic}).

Let $\A=(Q,I,T)$ be a finite automaton on the alphabet $A$
with $Q$ as set of states,
$I$ as set of initial
states and $T$ as set of terminal states. Such
an automaton is just a graph with $Q$ as set of vertices
and edges labeled by $A$.
The language \emph{recognized} by $\A$ is the set of labels
of paths from $I$ to $T$. The automaton is \emph{deterministic}
if for every $q\in Q$ and $a$ in $A$, there is at most
one edge starting at $q$ labeled $a$ (the term deterministic
corresponds to the term \emph{right-resolving} used
in particular in~\cite{LindMarcus2021}).


An automaton is \emph{unambiguous} for every $p,q\in Q$
and $w\in A^*$, there is at most one path from $p$ to $q$
labeled by $w$. An automaton is \emph{unambiguous for aperiodic points}
if for every aperiodic two-sided infinite word $x$ there is at most one path of
the automaton labeled by $x$.

\begin{proposition}\label{propositionInjective}
  Let $\sigma\colon A^*\to B^*$ be a morphism. If $\sigma$
  is recognizable for aperiodic points, either 
   $\sigma$ is periodic or $\sigma$ is injective.
\end{proposition}
\begin{proof}
Assume that  $\sigma$ is not periodic.
  If $\sigma$ is not injective,  there are some distinct $u,v\in A^*$
  such that $\sigma(u)=\sigma(v)=d$. Since
  $\sigma$ is not periodic, there is some $w$
  such that $c=\sigma(w)$
  and $\sigma(u)$ are not powers of the same primitive word.
  Then ${^\omega}c\cdot dc^\omega$ is aperiodic and has more than
  one centered $\sigma$-representation. Thus $\sigma$ is not recognizable
  for aperiodic points.
  \end{proof}

For every finite
set $U\subset B^+$, there is a particular automaton 
which recognizes  $U^*$, called the \emph{flower automaton}
of $U$. 
Its 
set of states is the set $Q$  defined as
\begin{displaymath}
Q=\{(u,v)\in B^+\times B^+\mid uv\in U\}\cup \{\omega\},
\end{displaymath}
where $\omega$ is a new element.
There are four type of edges labeled by $a\in B$
\begin{eqnarray*}
(u,av)&\longedge{a}&(ua,v)\quad\text{for $uav\in U$, $u,v\ne1$}\\
\omega&\longedge{a}&(a,v)\quad\text{for $av\in U$, $v\ne 1$}\\
(u,a)&\longedge{a}&\omega\quad\text{for $ua\in U$, $u\ne 1$}\\
\omega&\longedge{a}&\omega\quad\text{for $a\in U$.}
\end{eqnarray*}
The state $\omega$ is both initial and terminal.
It is easy to verify that the flower automaton recognizes $U^*$.
Moreover, the set $U$ is a code if and only if its flower automaton
is unambiguous (see~\cite{BerstelPerrinReutenauer2009}).

Note that a morphism $\sigma\colon A^*\to B^*$
is left marked if and only if the flower automaton of $\sigma(A)$
is deterministic.

\begin{example}
  Let $\sigma\colon a\mapsto ab,b\mapsto a$ be the Fibonacci
  morphism. The flower automaton $\A$ of $\sigma(A) =
  \{ab, a\}$
  is represented in Figure~\ref{figureFlowerAutomaton}.
  \begin{figure}[hbt]
    \centering
\tikzset{node/.style={circle,draw,minimum size=0.6cm,inner sep=0pt}}
\tikzset{title/.style={minimum size=0.5cm,inner sep=0pt}}
\tikzstyle{every loop}=[->,shorten >=1pt,looseness=8]
        \tikzstyle{loop above}=[in=120,out=60,loop]
    \begin{tikzpicture}
      \node[node](1)at(0,0){$\omega$};\node[title](it)at(-1,0){};
      \node[node](2)at(2,0){$(a,b)$};
      \draw[bend left,above,->,>= stealth](1)edge node{$a$}(2);
      \draw[above,->,>= stealth](1)edge[loop]node{$a$}(1);
      \draw[bend left,above,->,>= stealth](2)edge node{$b$}(1);
      \end{tikzpicture}
\caption{The flower automaton of $\sigma(A) =
  \{ab, a\}$.}\label{figureFlowerAutomaton}
  \end{figure}
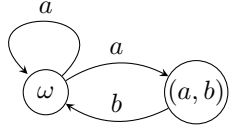
  
\end{example}

Let $\A$ be a finite automaton.
For $w\in A^*$, we denote by
$\varphi_\A(w)$  the  $Q\times Q$-relation defined by
\begin{displaymath}
\varphi_\A(w)_{p,q}=\begin{cases}1&\mbox{ if $p\edge{w}q$},\\0&\mbox{otherwise.}
\end{cases}
\end{displaymath}
The monoid $M(\A)=\varphi_\A(A^*)$ is the \emph{monoid of transitions} of the
automaton $\A$. For $m\in M(\A)$, we denote indifferently
$m_{p,q}=1$, $(p,q)\in m$  or $p\edge{m}q$ the fact
that $(p,q)$ are in relation by $m$.

When $\A$ is deterministic, each relation $\varphi_\A$
is a partial map from $Q$ to itself and thus $M(\A)$
is a submonoid of the monoid of partial maps from $Q$ to $Q$.

When $\A$ is unambiguous, the monoid $M=M(\A)$ is called
an \emph{unambiguous monoid of relations}. This
means that for every $m,n\in M$ and $p,q\in Q$
such that $p\edge{mn}q$ there is a unique $r\in Q$
such that $p\edge{m}r\edge{n}q$. This also means
that we can consider the elements of $M$ as $\{0,1\}$-matrices
with $0,1$ considered as integers.

As in any monoid, the Green relations $\GR$ and $\GL$
are defined by $m\GR n$ if $mM=nM$ and, symmetrically, $m\GL n$
if $Mm=Mn$. 
It is classical that $\GR$ and $\GL$ commute and thus that
the composition  $\GR\GL=\GL\GR$ is an equivalence, traditionally
denoted by $\GD$.

For an element $m$ of a monoid $M$, we denote by $H(m)$
the $\GH$-class of $m$, where $\GH$ is the Green relation
$\GH=\GR\cap\GL$. It is a group if and only if it
contains an idempotent $e$ (see~\cite{BerstelPerrinReutenauer2009}). In this case, every
$m\in H(e)$ has a unique inverse $m^{-1}$ in the
group $H(e)$.

Let $M$ be a unambiguous monoid of relations on $Q$.
A \emph{fixed point} of a relation $m\in M$  is
an element $q\in Q$ such that $q\longedge{m}q$.
A group in a monoid $M$ is a subsemigroup of $M$ which is a group.
Let $G$ be a group in $M$ with neutral element $e$. Let $S$ be the set
of fixed points of $e$. The restriction of the elements
of $G$ to $S\times S$ is a faithful representation
of $G$ by permutations on $S$. The number of elements
of $S$ is called the \emph{degree} of $G$.

The above representation of groups also exists in a monoid
of relations which is ambiguous but is more complicated (see
\cite{PerrinRyzhikov2021} or \cite{PerrinRyzhikov2020}).



A group in the monoid  $M(\A)$ is \emph{strongly cyclic}
(or \emph{special}) if $\varphi_\A^{-1}(G)$ is a cyclic submonoid of $A^*$.
This notion was introduced by Sch\"utzenberger in \cite{Schutzenberger1979}
(see also \cite{PerrinRindone2003}).
It implies that $G$ itself is a cyclic group. It also
implies that $\varphi^{-1}(D)$, where $D$ is
the $\GD$-class containing $G$, is also
included in the set of factors of the same cyclic submonoid.
In particular, if $e,f$ are idempotents in $D$, then
$\varphi_\A^{-1}(e)=u^+$ and $\varphi_\A^{-1}(f)=v^+$
where $u,v$ are conjugate. Moreover, if $m\in D$
is such that $m=em=mf$, then
\begin{equation}
  \varphi_\A^{-1}(e)=(rs)^+,\ \varphi_\A^{-1}(m)=(rs)^+r,\ \varphi_\A^{-1}(f)=(sr)^+
  \label{eqDclass}
  \end{equation}


A pair $(m,e)$ of elements in a monoid $M$ is a \emph{linked pair}
if $m=me$ and $e=e^2$.
The following result is \cite[Theorem 2.2]{PerrinPin2004}.
\begin{proposition}\label{propositionemf}
  Let $\varphi:A^*\to M$ be  a morphism from $A^*$ into a finite
  monoid $M$. For every $x\in A^\N$, there exists a linked pair $(m,e)$
  in $M$ such that $x\in\varphi^{-1}(m)\varphi^{-1}(e)^\omega$.
  \end{proposition}
A dual statement holds for a left infinite sequence. Consequently,
for every $x\in A^\Z$, there is a triple $(e,m,f)$ with
$e,f$ idempotents and $m=em=mf$ such that
$x\in{^\omega}\varphi^{-1}(e)\varphi^{-1}(m)\varphi^{-1}(f)^\omega$.

\begin{example}
  Let $\sigma$ be the Fibonacci morphism and let
  $\A$ be the flower automaton of $\sigma(A)=\{ab,a\}$
  (see Figure~\ref{figureFlowerAutomaton}).

  Set $\alpha=\varphi_\A(a)$ and $\beta=\varphi_\A(b)$.
  One has $a^2\equiv a$,
  $aba\equiv a$ and $b^2\equiv b^3$,
  which is a zero of the monoid.
  The monoid of transitions of $\A$,  is formed of $6$ elements
  pictured in Figure \ref{figureSMF}
  (we adopt the usual egg-box representation of monoids
  corresponding to the Green relations, see~\cite{BerstelPerrinReutenauer2009}).
  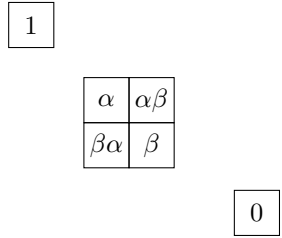
\begin{figure}[hbt]
    \centering
\tikzset{node/.style={draw,minimum size=0.6cm,inner sep=0pt}}
\tikzset{title/.style={minimum size=0.5cm,inner sep=0pt}}
    \begin{tikzpicture}
      \node[node](1)at(-1,2){$1$};
      \node[node](a)at(0,1){$\alpha$};
      \node[node](ab)at(0.6,1){$\alpha\beta$};
      \node[node](ba)at(0,.4){$\beta\alpha$};
      \node[node](ab)at(0.6,.4){$\beta$};
      \node[node](0)at(2,-.5){$0$};
      \end{tikzpicture}
\caption{The  monoid of transitions of $\A$.}\label{figureSMF}
  \end{figure}
  Besides $1$ and $0$, there are three groups in $M(\A)$
  reduced to one element ($\alpha$, $\alpha\beta$ and $\beta\alpha$), which have
  degree $1$.
\end{example}

\begin{theorem}\label{theoremSyntactic}
  Let $\sigma\colon A^*\to B^*$ be an injective morphism, and let  $\A$ be the flower automaton of
  $\sigma(A)$.
  The morphism $\sigma$ is recognizable for aperiodic points
  if and only if  $\A$ is unambiguous for aperiodic points.
\end{theorem}
\begin{proof}
 The morphism $\sigma$ is
recognizable for aperiodic points if and only if every aperiodic point
$x$ of $B^\Z$ has at most one decomposition in words of $\sigma(A)$, that is,
if there is at most one two-sided infinite path of $\A$ labeled by $x$ going
through the state $\omega$ for an infinite number of positive and
negative indices. As $\sigma(A)$ is finite, every two-sided infinite path of $\A$ goes
through the state $\omega$ for an infinite number of positive and
negative indices. Thus $\sigma$ is
recognizable for aperiodic points if and only $\A$ is unambiguous for aperiodic points.
\end{proof}

An automaton $\A$ is \emph{weakly deterministic} if there
is an integer $n$ such that for every state $p$ and
every word $w$ of length $n$, there is at most one right infinite path
starting at $p$ with label starting by $w$. Symmetrically,
$\A$ is \emph{weakly codeterministic} if there
is an integer $n$ such that for every state $p$ and
every word $w$ of length $n$, there is at most one left infinite path
ending at $p$ with label ending by $w$.

The notion of weakly deterministic automaton corresponds
to that of \emph{right closing} map in symbolic dynamics.
More precisely, the map assigning its label  to a right
infinite path in the automaton is right closing if
and only if the automaton is weakly deterministic.

The following statement gives a necessary condition for
an automaton to be unambiguous for aperiodic points.
An automaton is \emph{periodic} if the labels all infinite paths
are periodic. Otherwise, it is said to be \emph{non-periodic}.
Note that in a periodic automaton which is strongly connected
the labels of all finite paths are factors of $u^*$ for a single word $u$.
\begin{proposition}\label{propWeakly}
  If a strongly connected non-periodic
  finite automaton is unambiguous for aperiodic points,
  it is weakly deterministic and co-deterministic.
\end{proposition}
\begin{proof}
  Assume that $\A$ is not weakly deterministic. Then there
  is a state $p$ and a right-infinite sequence $x$
  such that there are two paths starting at $p$
  labeled $x$. Since $\A$ is strongly connected,
  there is a left-infinite path ending at $p$.
  Let $y$ be  the
  label of this path. Since $\A$ is non-periodic, we may
  choose $y$ aperiodic. Then $y\cdot x$ is
  an aperiodic point which is the label of more than
  one path in $\A$. The case where $\A$ is
  not weakly co-deterministic is symmetrical.
  \end{proof}

We now characterize as follows the unambiguous automata
which are unambiguous for aperiodic points. Combined with
Proposition~\ref{propositionInjective}
and Theorem~\ref{theoremSyntactic}, this will give us
a characterization of morphisms which are recognizable
for aperiodic points. Indeed, if $\sigma$ is
recognizable for aperiodic points,
by Proposition~\ref{propositionInjective} either it is
periodic or it is injective. In the first case,
it is trivially recognizable for aperiodic points.
In the second case, by Theorem~\ref{theoremSyntactic}, it is recognizable
for aperiodic points if and only if the flower automaton
of $\sigma(A)$ is unambiguous for aperiodic points.

\begin{theorem}\label{theoremSyntactic2}
  A strongly connected unambiguous finite automaton $\A$
  is unambiguous for aperiodic points if and only if the
  following conditions are satisfied.
  \begin{enumerate}
    \item[\rm(i)] $\A$ is weakly deterministic and co-deterministic.
    \item[\rm(ii)] For every idempotent $e\in\varphi_\A(A^+)$
      of rank at least $2$, the
    group $H(e)$ is strongly cyclic.
  \item[\rm(iii)] For every pair $(e,f)$ of idempotents in $\varphi_\A(A^+)$
    of rank at least $2$, and every pair of distinct fixed points $p,p'$
    and $q,q'$ of $e$ and $f$ respectively,
    if  some $m\in\varphi_\A(A^*)$ is such that
    $p\edge{m}q$ and $p'\edge{m}q'$, then
    $e,f$ and $m$ belong to the same $\GD$-class.
  \end{enumerate}
\end{theorem}
\begin{proof}
  Assume that $\A$ is unambiguous for aperiodic points.
  By Proposition~\ref{propWeakly}, condition (i) is satisfied.
  Let $e$ be an idempotent in $M(\A)$ of degree $d\ge 2$.
  If the group  $G=H(e)$ is not strongly cyclic,
   there are words of arbitrary large minimal period
  in $L=\varphi_\A^{-1}(G)$ and thus there is an aperiodic 
  sequence $x\in A^\Z$ such that all its factors are factors of words in $L$.
  Indeed, if $G$ is not strongly cyclic, there
  are words $u,v$ which are not powers of the same
  primitive word such that $e\in\varphi_\A(u)\cap\varphi_\A(v)$.
  Then $\{u,v\}^*$ contains words of arbitrary large minimal period
  and thus an infinite sequence $x$ with all its factors
  which are factors of words in $\{u,v\}^*$.
  Then $x$ is the label of $d$ paths in $\A$ and consequently,
  $\A$ is not recognizable for aperiodic points. This proves
  that condition (ii) holds. Assume finally that $m$
  is such that  $p\edge{m}q$ and $p'\edge{m}q'$
  for some fixed points $p,p'$ and $q,q'$ of
  two idempotents $e,f\in\varphi(A^+)$. If $e,f$ do not belong to the
  same $\GD$-class, we have $e\in\varphi_\A(u^+)$
  and $f\in\varphi_\A(v^+)$ for non-conjugate primitive words $u,v$
  (two idempotents are in the same $\GD$-class if and only if
  they are conjugate).
  Let $w$ be such that $\varphi(w)=m$. Then ${^\omega}u w v^\omega$
  is an aperiodic sequence which is the label of two distinct paths.

  Assume conversely that the conditions are satisfied. Consider
  a point $x\in A^\Z$ which is the label of more than one path.
  Since $M=M(\A)$ is a finite
  monoid, there is by Proposition \ref{propositionemf},
  some $m\in M$ and idempotents $e,f\in M$
  with $em=mf=m$ such that $x\in {^{\omega}EwF^\omega}$ with
  $E=\varphi_\A^{-1}(e)$, $\varphi_\A(w)=m$ and $F=\varphi_\A^{-1}(f)$.
  The two paths have
  the form $p\edge{e}p\edge{m}q\edge{f}q$
  and $p'\edge{e}p'\edge{m}q'\edge{f}q'$. If $p=p'$,
  then $\A$ is not weakly deterministic. Similarly,
  if $q=q'$, it is not weakly co-deterministic. Thus, by condition (iii),
  $e,m,f$ belong to the same $\GD$-class. By condition
  (ii), using Equation~\eqref{eqDclass}, this implies that $x$ is periodic.
  
\end{proof}

\begin{example}\label{exampleab,ba}
  Let $\sigma\colon a\mapsto ab,b\mapsto ba$ be the Thue-Morse morphism
  and let $\sigma(A)=\{ab,ba\}$. The flower automaton of $\sigma(A)$
  is represented in Figure~\ref{figureMinimalAutomatonTM}.
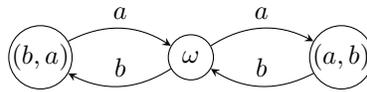
\begin{figure}[hbt]
    \centering
\tikzset{node/.style={circle,draw,minimum size=0.6cm,inner sep=0pt}}
\tikzset{title/.style={minimum size=0.5cm,inner sep=0pt}}
\tikzstyle{every loop}=[->,shorten >=1pt,looseness=8]
        \tikzstyle{loop above}=[in=120,out=60,loop]
    \begin{tikzpicture}
      \node[node](1)at(0,0){$(b,a)$};
      \node[node](2)at(2,0){$\omega$};\node[title](it)at(2,1){};
      \node[node](3)at(4,0){$(a,b)$};
      \draw[bend left,above,->,>= stealth](1)edge node{$a$}(2);
      \draw[bend left,above,->,>= stealth](2)edge node{$b$}(1);
      \draw[bend left,above,->,>= stealth](2)edge node{$a$}(3);
      \draw[bend left,above,->,>= stealth](3)edge node{$b$}(2);
      \end{tikzpicture}
\caption{The flower automaton of $\sigma(A)=\{ab,ba\}$.}\label{figureMinimalAutomatonTM}
\end{figure}
  The  monoid of transitions of $\A$ is
  represented in Figure~\ref{figureSMTM} with $\alpha=\sigma(a)$
  and $\beta=\sigma(b)$.
  \begin{figure}[hbt]
    \centering
\tikzset{node/.style={draw,minimum size=0.6cm,inner sep=0pt}}
\tikzset{title/.style={minimum size=0.5cm,inner sep=0pt}}
    \begin{tikzpicture}
      \node[node](1)at(-1,2){$1$};
      \node[node](a)at(0,1){$\alpha$};
      \node[title]at(0,1.5){$2,3$};
      \node[node](ab)at(0.6,1){$\alpha\beta$};
      \node[node](ba)at(0,.4){$\beta\alpha$};\node[node](ab)at(0.6,.4){$\beta$};
      \node[title]at(2,.5){$3$};
      \node[node](aa)at(2,0){$\alpha^2$};
      \node[node](a^2b)at(2.6,0){$\alpha^2\beta$};
      \node[node](a^2b^2)at(3.2,0){};
      \node[node](ba^2)at(2,-.6){$\beta\alpha^2$};
      \node[node](ba^2b)at(2.6,-.6){};
      \node[node](ba^2b^2)at(3.2,-.6){};
      \node[node](b^2a^2)at(2,-1.2){};
      \node[node](b^2a^2b)at(2.6,-1.2){};
      \node[node](b^2a^2b^2)at(3.2,-1.2){$\beta^2$};
      \node[node](0)at(4.5,-2){$0$};
      \end{tikzpicture}
\caption{The  monoid of transitions of $\A$.}\label{figureSMTM}
  \end{figure}
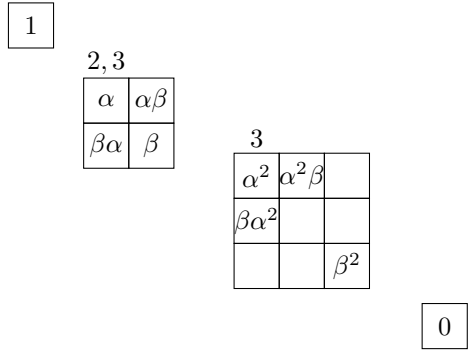
  There are two groups of degree $2$, each reduced to one element,
  namely $ab$ and $ba$. They are strongly cyclic since
  $\varphi_{\A}^{-1}(\alpha\beta)=(ab)^*$.
\end{example}

\section{Efficient algorithms} \label{section.efficient}
The conditions of Theorem \ref{theoremSyntactic2} give an effective
characterization of morphisms recognizable for aperiodic points.
We now consider the problem of finding a more efficient procedure
to check whether this property holds.

If $\A=(Q,I,T)$ is a finite automaton on $A$, the
\emph{square} $\A\times\A$ of the automaton $\A$ is
the following automaton. Its  set of states is $Q\times Q$ and
its edges are the $((p,q),a,(r,s))$ such that
$(p,a,r)$ and $(q,a,s)$ are edges of $\A$.

The \emph{diagonal} of $\A\times\A$ is the set of states of the
form $(p,p)$ with $p\in Q$. A strongly connected component of
an automaton is \emph{non-trivial} if it contains at least one cycle.

A \emph{simple path} in $\A$ is a path with all states distinct
except the extremities.
A \emph{simple cycle} around a state $p$ of $\A$ is a simple path from $p$ to
$p$.
A non-trivial strongly connected component of
an automaton is said to be reduced to a cycle if there is a simple cycle
$\pi_p$ around each state $p$ of the component such that any other cycling path
around $p$ is a power of $\pi_p$.

It is easy to verify that $\A$ is unambiguous if and only if there is no
path in $\A\times\A$ going from a diagonal state to a diagonal state
and going through a non-diagonal one.
This gives a polynomial algorithm to check whether an automaton
is unambiguous.

It is also easy to verify that if $\A$ is unambiguous,
its square is also unambiguous. Indeed, assume
that there are two paths in $\A \times \A$
\begin{align*}
  &(p,q)\edge{w}(u,v)\edge{w'}(r,s),\\
  &(p,q)\edge{w}(u',v')\edge{w'}(r,s),
\end{align*}
then there are paths 
$p\edge{w}u\edge{w'}r$,  $p\edge{w}u'\edge{w'}r$
and 
$q\edge{w}v\edge{w'}s$,  $q\edge{w}v'\edge{w'}s$ in $\A$
implying $u = u'$ and $v =v'$.

We have already seen in Proposition~\ref{propositionInjective}
that if $\sigma$ is a morphism
which is recognizable for aperiodic points,
then it is injective or periodic.

The size of an automaton is the sum of the number of states and the number
of edges of the automaton.
\begin{proposition} \label{proposition.unambiguous.aperiodic}
Let $\A$ be an unambiguous strongly connected automaton. Then $\A$ is
unambiguous for aperiodic points if and only if 
the following conditions are satisfied.
\begin{enumerate}
\item[\rm(i)] Every non-trivial strongly connected component of
  $\A\times \A$ out of the diagonal is reduced to a cycle.
\item[\rm(ii)] There is no path between two non-trivial
  strongly connected components of $\A\times \A$.
\end{enumerate}
All these conditions can be checked in a quadratic time in the size of $\A$.
\end{proposition}
\begin{proof}
Since $\A$ is unambiguous, $\A \times \A$ is unambiguous. Note that the set
of diagonal states is a non-trivial strongly connected component of $\A
\times \A$. It cannot contain non-diagonal states since $\A$ is
unambiguous.

We first show that the conditions are necessary.
Let $C$ be a non-trivial strongly connected component of $\A \times
\A$ which is non-diagonal. Let $\pi$ be a simple cycle
labeled by $u$ around $(p,q)$ with $p \neq q$
in $C$ and another cycle $\pi'$ which is not a power of $\pi$ around $(p,q)$
labeled by $v$.
Since $\A \times \A$ is unambiguous, 
$u$ and $v$ are not a power of the same
word. Thus there is a
two-sided infinite sequence $x$ which is a concatenation of words $u$ and $v$
which is not periodic. This sequence
is the label of a two-sided infinite path of
$\A \times \A$ going through $(p,q)$ after reading $u$ or $v$.
Since $p \neq q$ this leads to two distinct paths labeled by $x$
in $\A$ and thus $\A$ is not
unambiguous for aperiodic points.

Let $C$ and $C'$ be two non-trivial non-diagonal strongly connected
components of $\A \times \A$. Then $C$ and $C'$ are reduced to a
cycle. Let $(p,q) \in C$ with $p \neq q$ and
$(r,s) \in C'$ with $r \neq s$, let $\pi$ be a simple cycle  labeled by
$u$ around $(p,q)$, and $\pi'$ be a simple cycle  labeled by
$v$ around $(r,s)$.
Let us assume that there is a path from
$(p,q)$ to $(r,s)$ labeled by $w$ in $\A \times \A$.
Assume that $u$ and $v$ are not powers of two conjugate words, then
the two-sided infinite word $x = {^\omega}uwv^\omega$ is aperiodic and
$\A$ is not unambiguous for aperiodic points.
Assume now that $u=(rs)^k$ and $v = (sr)^{k'}$ with $rs$ primitive.
The word $x = {^\omega}uwv^\omega$ is periodic if and only if $w \in
(rs)^*r$. But in this case, if $w = (rs)^mr$, the word $(rs)^{kk'}
(rs)^m r = (rs)^m r (sr)^{kk'}$ is the label of two distinct paths
from $(p,q)$ to $(r,s)$ which is impossible since $\A \times \A$ is
unambiguous. 

Let us now assume that the strongly connected component of diagonal
points is connected to a non-trivial non-diagonal strongly connected component.
For instance if there
is a path labeled by $w$ from a state $(p,p)$ to a state $(q,r)$ with $q \neq r$
belonging to an non-trivial strongly connected component.

Assume that the graph of $\A$ is not reduced to one cycle. Then there
are a simple cycle $\pi$ around $(p,p)$ labeled $u$, another cycle
$\pi'$ around $(p,p)$ labeled $v$ which is not a power of $\pi$.
The words $u$ and $v$ are not a power of a same
word. 
Hence there is a left-infinite word $y$ being a concatenation
of words $u$ and $v$ labeling a path
ending in $(p,p)$ and which is not ultimately periodic.
Let $z$ be the label of a cycle around $(q,r)$. Then 
$ywz^\omega$ is an
aperiodic point labeling two distinct paths of $\A$.

If $\A$ is reduced to one cycle, there is no path from a diagonal state to a non-diagonal
one in $\A \times \A$.
Thus if condition \rm(ii) is not satisfied, $\A$ is
not unambiguous for aperiodic points.

We now show that the conditions are sufficient. We show that if all
conditions hold, $\A$ is unambiguous for aperiodic points.
Let $x$ be an aperiodic point which is the label of two distinct paths
of $\A$. Then there is a two-sided infinite
path $(p_i,q_i)_{i \in \Z}$ in $\A \times \A$ labeled by $x$ such
that $(p_i)_{i \in \Z} \neq (q_i)_{i \in \Z}$. Without loss of generality, we may assume that $p_0 \neq
q_0$. There is an infinite number of $i <0$
such that $(p_i,q_i) = (p,q)$ for some states $p, q$. There is also an
infinite number of $j > 0$
such that $(p_j,q_j) = (r,s)$ for some states $r, s$. We cannot have
$p = q$ and  $r = s$ since $\A$ is unambiguous.

Assume that $p = q$. Then $ r \neq s$. Then there is a path from
$(p,p)$ to the non-trivial strongly connected component of $(r,s)$ which is
forbidden. Similarly it is impossible to have $r = s$ and $p \neq q$.

Assume now that $p \neq q$ and $r \neq s$. If the strongly connected
components of $(p,q)$ and $(r,s)$ are distinct, then there is a path
between these two non-trivial components which is forbidden.  If the components of
$(p,q)$ and $(r,s)$ are the same, then $x$ is periodic, a contradiction. 

The conditions can be checked in time $O(n^2)$ where $n$ is the size
of $\A$.
\end{proof}
\begin{example}\label{exampleProposition8.1}
  Let us consider again the morphism $\sigma\colon a\mapsto aa, b\mapsto ab,
  c\mapsto ba$ of Example~\ref{example3.4}. The flower automaton $\A$
  of $\sigma(A)$ is represented in Figure~\ref{figureExample8.1}
  on the left. The non-diagonal part of the automaton $\A\times\A$
  has a strongly connected component, which
  is represented on the right. This component
  is not reduced to a cycle and thus condition (i)
  is not satisfied, consistently with the fact that $\sigma$
  is not recognizable for aperiodic points, as we have seen in
  Example~\ref{example3.4}.
  \begin{figure}[hbt]
    \centering
\tikzset{node/.style={circle,draw,minimum size=0.6cm,inner sep=0pt}}
\tikzset{title/.style={minimum size=0.5cm,inner sep=0pt}}
\tikzstyle{every loop}=[->,shorten >=1pt,looseness=8]
        \tikzstyle{loop above}=[in=120,out=60,loop]
    \begin{tikzpicture}
      \node[node](1)at(0,0){$1$};
      \node[node](2)at(2,0){$\omega$};
      \node[node](3)at(4,0){$3$};
      \node[node](4)at(2,2){$2$};
      \draw[bend left,above,->,>= stealth](1)edge node{$a$}(2);
      \draw[bend left,above,->,>= stealth](2)edge node{$a$}(1);
      \draw[bend left,above,->,>= stealth](2)edge node{$b$}(3);
      \draw[bend left,above,->,>= stealth](3)edge node{$a$}(2);
      \draw[bend left,left,->,>= stealth](2)edge node{$a$}(4);
      \draw[bend left,right,->,>= stealth](4)edge node{$b$}(2);

      \node[node](3omega)at(6,0){$3,\omega$};
      \node[node](omega2)at(6,2){$\omega,2$};
      \node[node](omega1)at(8,0){$\omega,1$};
      \node[node](1omega)at(8,2){$1,\omega$};
      \node[node](2omega)at(10,0){$2,\omega$};
      \node[node](omega3)at(10,2){$\omega,3$};

      \draw[bend left,left,->,>= stealth](3omega)edge node{$a$}(omega2);
      \draw[bend left,right,->,>= stealth](omega2)edge node{$b$}(3omega);
      \draw[above,->](3omega)edge node{$a$}(omega1);
      \draw[above,->](1omega)edge node{$a$}(omega2);
      \draw[bend left,left,->,>= stealth](omega1)edge node{$a$}(1omega);
      \draw[bend left,right,->,>= stealth](1omega)edge node{$a$}(omega1);
      \draw[above,->](omega1)edge node{$a$}(2omega);
      \draw[above,->](omega3)edge node{$a$}(1omega);
      \draw[bend left,left,->,>= stealth](2omega)edge node{$b$}(omega3);
      \draw[bend left,right,->,>= stealth](omega3)edge node{$a$}(2omega);
      \end{tikzpicture}
\caption{The automata $\A$ and $\A\times\A$.}\label{figureExample8.1}
    \end{figure}
\end{example}
Note that Proposition~\ref{proposition.unambiguous.aperiodic}
gives an easy proof of Proposition~\ref{lemmaLeftMarked}.
Indeed, if $\sigma\colon A^*\to B^*$ is left marked, 
let $\A$ be the flower automaton of $\sigma(A)$. Since $\sigma$
is left marked, $\A$ is deterministic and thus $\A \times \A$ also.
Note that if $p \neq \omega$ there is at most one edge going out of
$p$ in $\A$.

Thus if $(p,q)$ is a state of $\A \times \A$ with $p$ or $q$ distinct from
$\omega$, there is at most one edge going out of $(p,q)$ in $\A \times
\A$. Further, since $\A$ is deterministic, there is no edge
from a diagonal state to a non-diagonal one in $\A \times \A$.

As a consequence each non-trivial strongly connected component of
$\A\times \A$ out of the diagonal is reduced to a cycle and there is
no path between two non-trivial strongly connected components of
$\A \times \A$. Hence $\sigma$ is recognizable for aperiodic points.

We define the size of a finite set of words as the sum of the lengths of words
of the set.
 \begin{corollary}\label{corollary.recognizable}
   It can be checked in quadratic time in the size of $\sigma(A)$
   whether an injective morphism $\sigma \colon A^* \rightarrow B^*$ is
   recognizable for aperiodic points.
 \end{corollary}
 \begin{proof}
   The flower automaton of $\sigma(A)$ is unambiguous
   and strongly connected. The result follows directly from  Proposition
   \ref{proposition.unambiguous.aperiodic}. The algorithm runs in time
   $O(n^2)$ where $n$ is the size of $\sigma(A)$ that is the sum of the
   lengths of the words of $\sigma(A)$.
 \end{proof}
 \section{Acknowledgments}
 We thank Francesco Dolce for his careful reading of the manuscript
 and the referee for useful comments.

\bibliographystyle{plain}
\bibliography{recognizability}
\end{document}